\documentclass[12pt, english, a4paper]{amsart}
\usepackage[margin=2.3cm]{geometry}
\usepackage{amssymb,amsmath,amsthm}
\usepackage{commath}
\usepackage{mathtools}
\usepackage{mathrsfs}
\usepackage{accents}
\usepackage[utf8]{inputenc}
\usepackage[T1]{fontenc}

\usepackage{caption}
\usepackage{subcaption}

\usepackage{graphicx}
\usepackage{xcolor}
\usepackage{pgf, tikz}
\usepackage{float}
\usepackage{tikzlings}
\usepackage{tikzducks}

\usepackage[colorlinks=true, allcolors=black]{hyperref}

\usepackage{url}

\usepackage{datetime}

\usepackage{chngcntr}

\theoremstyle{plain}
\newtheorem{theorem}{Theorem}
\newtheorem{definition}{Definition}
\newtheorem{lemma}{Lemma}
\newtheorem{conjecture}{Conjecture}

\theoremstyle{remark}
\newtheorem{remark}{Remark}

\newcommand{\R}{\mathbb{R}}
\newcommand{\M}{\mathcal{M}}
\newcommand{\eps}{\varepsilon}
\newcommand{\bd}{\partial}
\newcommand{\zz}{\mathbf{z}}

\date{\today}

\thanks{Research was partially supported by the ERC Advanced Grant "GeoScape" no.  882971, by Hungarian National Research (NKFIH) grants no. 147145, 147544, and 150151, which has been implemented with the support provided by the Ministry of Culture and Innovation of Hungary from the National Research, Development and Innovation Fund, financed under the ADVANCED-24 funding scheme. This research was funded by the grant 2024-1.2.8-TÉT-IPARI-CN-2025-00011,
with the support provided by the National Research,
Development and Innovation Office from the National Research,
Development and Innovation Fund, and financed under the
2024-1.2.8-TÉT-IPARI-CN funding scheme.
}


\title{Softening locally polyhedral tilings}   
\author{Gergely Ambrus}
\author{Dorottya Dancsó}

\begin{document}

\begin{abstract}
We call a cell $C \subset \mathbb{R}^d$ {\em soft} if every point of its boundary lies on a smooth curve contained in $\bd C$. A tiling of the space is called {\em completely soft} if all of its cells are soft. In their 2024 article, Domokos, Goriely, G. Horv\'{a}th and Reg\H{o}s conjectured that every polyhedral tiling of $\mathbb{R}^3$ satisfying mild regularity assumptions can be locally deformed into a completely soft tiling. By constructing an algorithm that first bends the edges emanating from each vertex and then extends this transformation to a sufficiently smooth deformation, they proved the conjecture for  polyhedral tilings satisfying a certain combinatorial condition. In the present paper, we precisely describe a new edge-bending algorithm that establishes a more general version of this conjecture: every {\em locally polyhedral} tiling of $\mathbb{R}^3$ can be completely softened.  We also give a short proof of an earlier result of Domokos, G. Horv\'{a}th, and Reg\H{o}s stating that every suitably nondegenerate polygonic tiling of the plane has, on average, at least two points per cell at which the softness criterion is violated.
\end{abstract}


\maketitle

\section{Overview}\label{sec:tilings}

Space-filling shapes have long captivated mathematicians; the theory of tilings (or tessellations) has remained a vibrant and actively researched area of geometry from Plato to the present day \cite{grunbaum1987tilings, harriss2017tilings, Schulte1993Tilings}. Beyond their architectural and artistic applications, space-filling patterns frequently occur in nature, for example, in the internal structure of crystals and foams. Consequently, the study of tilings has attracted sustained attention across a wide range of natural sciences, with a primary focus on the two- and three-dimensional settings. Although the planar case is now well understood and comparatively easier to visualize, the three-dimensional realm of space-filling shapes is richer and considerably more complex. This is the central topic of the article.

A {\em tiling} of $\R^d$ is a covering of $\R^d$ with cells that do not overlap (i.e., which are disjoint in their interior), and whose union is the whole space $\R^d$. Typically, cells are assumed to be {\em topological $d$-balls}: homeomorphic images of a closed Euclidean ball in $\R^d$. The most extensively studied tilings partition the plane into convex polygons, or tessellate $\R^3$ by convex polyhedra; these form the {\em cells} of the tiling. Throughout the paper, we will assume that all polygons and polyhedra are convex. Accordingly, we will omit the reference to convexity.

In a {\em polyhedral tiling} of $\R^3$, we require that the cells be convex polyhedra,  any two of which intersect only in common vertices, edges, or faces (thus, e.g., no cell vertex may lie in the relative interior of an edge of another cell). Polygonal tilings of the plane are defined analogously.
Furthermore, a tiling is {\em normal} if the sizes of the cells are uniformly bounded, see Definition~\ref{def:tiling} -- this also implies that the number of cells containing any given point is finite and uniformly bounded. We always impose this normality condition. The vertices of the cells are called the {\em nodes} of the tiling. 

Of particular interest are {\em monohedral tilings}, in which all cells are congruent; their common cell is called a {\em monotile}. Despite many classical constructions, determining which shapes admit monohedral tilings remains a difficult classification problem. Although numerous space-filling shapes and several infinite families are known, no general classification is available, even within important classes such as tetrahedra. Monotiles satisfying additional geometric requirements are therefore especially noteworthy.

Polyhedral tilings admit various generalizations. Following the work of Domokos, Goriely, G. Horváth, and Regős~\cite{DOM2024}, our focus is on  {\em polygonic} and {\em polyhedric tilings}: here, $\R^d$ for $d=2$ or $d=3$, respectively, is tiled with not necessarily convex cells that are compact topological $d$-balls whose pairwise intersections are lower-dimensional topological balls,  and which satisfy additional combinatorial and smoothness conditions (see Section~\ref{sec:defi} for the formal definition). Vertices, edges, and -- in the three-dimensional case -- faces of the tiling are defined according to their topological dimension, in direct analogy with the polygonal and polyhedral cases. The {\em nodes} of the tiling play a central role: these are points at which the number of containing cells is strictly locally maximal. 

We are mainly interested in tilings whose cells satisfy specific smoothness conditions. 
Here, {\em smoothness} is understood to mean the $C^2$ property (i.e., differentiability twice continuously). If necessary, all subsequent arguments can be readily adapted to accommodate higher degrees of smoothness, including infinite differentiability.

Given a topological ball $C$ and a point $p \in \bd C$ on its boundary, $p$ is called a {\em spike} (or sharp corner) of $C$ if there is no smooth curve on $\partial C$ that goes through $p$. The cell $C$ is called {\em soft} if it does not have spikes. Finally, a tiling is {\em completely soft} if it is composed of soft cells. 
These notions apply in both two- and three-dimensional settings. Note that our definition of softness varies from that given in \cite{DOM2023, DOM2024}; see the remark after Definition~\ref{def:softcell}.

The first question that arises in the context of soft cells is whether there are completely soft polygonic and polyhedric tilings of $\R^d$ for $d = 2,3$, respectively. In 2023, Domokos, G. Horváth, and Regős \cite{DOM2023} resolved the planar case, showing that in any polygonic tiling of the plane, the cells must have at least two spikes on average, provided that the average exists. However, the three-dimensional case is markedly different: there exist many polyhedric tilings consisting entirely of soft cells. The focus here is on whether, starting from a polyhedral tiling, one can modify it so that all cells become smooth.

Many  polyhedric tilings can be obtained from polyhedral tilings through local diffeomorphisms; however, this is not the case for all (see, e.g., Figure~\ref{fig:paplan}). We focus on a particular subclass: a polyhedric tiling is said to be {\em locally polyhedral} if every node is locally homeomorphic to a node of a polyhedral tiling (i.e., there is a homeomorphism between suitable neighborhoods of them in the two tilings). However, this local condition does not imply that the entire tiling is globally homeomorphic to a polyhedral tiling (see, e.g., Figure~\ref{fig:orsolakhely}).

The principal question we address is the following: can any locally polyhedral tiling be converted into a completely soft polyhedric tiling by a family of local homeomorphisms?  In brief terms, {\em can any locally polyhedral tiling be completely softened}? 
Answering a special case of the question, Domokos, G. Horváth, Goriely, and Regős proved in \cite{DOM2024} that all polyhedral tilings satisfying a certain combinatorial condition can be completely softened. They further conjectured that this condition is superfluous, namely, that {\em every} polyhedral tiling can be completely softened.

The purpose of this paper is to verify this conjecture for a broader class of tilings. By constructing a new edge-bending algorithm, we prove that every locally polyhedral tiling of $\R^3$ can be completely softened.

\section{Terminology and results}\label{sec:defi}

Throughout the article, we work in the Euclidean space $\R^d$ where $d = 2,3$; these will be referred to as {\em plane} and {\em space} in the cases  $d=2$ and $d=3$, respectively. We will denote the origin by $o$.

A closed halfplane/halfspace is a closed subset of $\R^d$, where $d=2,3$, whose boundary is a line/plane. A \emph{convex polygon}/\emph{convex polyhedron} is defined as a compact set with a nonempty interior that can be represented as the intersection of finitely many closed halfplanes/halfspaces. We will omit the reference to convexity, calling these objects simply {\em polygons} in the plane and {\em polyhedra} in the three-dimensional case. The boundary of a polygon consists of vertices and edges in between, whereas the boundary of a polyhedron constitutes vertices, edges, and faces (sides); the latter are convex polygons. These boundary objects of the polyhedron, along with the inclusion relation, make up its {\em face lattice}. Two polyhedra are {\em combinatorially equivalent} if their face lattices are isomorphic. Since we do not consider polytopes of dimension greater than three, we will not use the term \emph{facet}. For more details on polyhedra, see \cite{gruber2007convex,grunbaum2013convex,ziegler2012lectures}.

Given a polyhedron $P$, its vertices and edges define a graph, called the  {\em edge graph} of~$P$, and denoted by $G(P)$. It is well-known that $G(P)$ is planar; moreover, Steinitz's theorem \cite{grunbaum2013convex} states that the class of edge graphs of polyhedra equals the class of undirected, 3-vertex-connected planar graphs. These, along with a specified planar representation, will be called \emph{polyhedral graphs}. By an inverse stereographic projection (or projecting $P$ centrally onto a sphere $S$ that is centered at an interior point of $P$), the graph $G(P)$ can also be interpreted as being drawn on a sphere~$S$. Thus, $G(P)$ induces a subdivision of the sphere, and we say that the graph $G(P)$ and the corresponding spherical complex are {\em{combinatorially equivalent}} to $P$: that is, there exists a bijection between the vertices, edges, and faces that preserves inclusion.

Next, we recall some basic notions of topology. As usual, $B^d = \{ x \in \R^d: |x| \leq 1\}$
denotes the $d$-dimensional closed unit ball. Moreover, for a point $z \in \R^d$ and a radius $r> 0$, let 
$ B^d(z, r) = \{ x \in \R^d: |x-z| \leq r\} $ be the $d$-dimensional ball centered at $z$ with radius $r$, whose boundary is the sphere $S^{d-1}(z, r)$. 

A set is called a \emph{$d$-dimensional topological ball} if it is homeomorphic to $B^d$. By a \emph{$d$-dimensional manifold} we mean a topological space in which every point has an  open neighborhood that is homeomorphic to the open unit ball $\mathrm{int} \, B^d$. The manifold is said to be of class $C^2$, or simply {\em smooth}, if the corresponding homeomorphisms can be chosen to form a twice continuously differentiable mapping. A smooth two-dimensional manifold will be referred to as a {\em smooth surface}. 

By a smooth curve, we mean the image $f([a,b])$ of a closed interval $[a,b] \subset \R$, $a \neq b$, by  a continuous, twice continuously differentiable,  injective mapping $f \colon [a,b] \to \mathbb{R}^3$ with $f'(x) \neq 0$ for any $x \in (a,b)$  and non-vanishing one-sided derivatives at the endpoints. The mapping $f$ is called the parametrization of the curve, which is regular if $f'(x) \neq 0$ holds  for any $x \in (a,b)$. In particular, we assume that the curve is simple, i.e., non-self-intersecting. The endpoints of the curve are $f(a)$ and $f(b)$, and its interior is a smooth 1-dimensional manifold. The \emph{half-tangent} of a smooth curve at an endpoint $p$ is defined as the ray starting at $p$ to which the rays $\overrightarrow{pq}$ converge as the point $q$ on the curve tends to $p$ -- it follows from the differentiability property that this indeed exists uniquely. Here and later on, for a ray starting at a point $p$ and passing through another point $q$, we use the notation $\overrightarrow{pq}$. 

Next, we introduce the core concepts of the paper. 
\begin{definition}\label{def:tiling}
     A tiling $\M$ of $\R^d$ is a family of $d$-dimensional topological balls, called \emph{cells}, whose interiors are pairwise disjoint and whose union is $\R^d$. The tiling is \emph{normal} if its cells are uniformly bounded; that is, there exist constants $0<r<R<\infty$ such that each cell contains a ball of radius $r$ and is contained in a ball of radius $R$.
\end{definition}

Cells are also called {\em tiles}. We note that since the rational points form a dense set in $\R^d$, any tiling is countable. Furthermore, a topological ball is called a {\em monotile} if the space can be tiled with its congruent copies.

A fundamental property of a normal tiling is that the number of cells containing a given point is finite and uniformly bounded: see Lemma~\ref{lemma:normaltiling} for a proof in the planar case, which extends to higher dimensions without difficulty.

Note that if all cells of a tiling are convex, then these must be polygons in the planar case and polyhedra in the three-dimensional setting.  This fact leads to a special subclass of tilings:
\begin{definition} \label{def:polihedralis}
    A tiling $\M=\{C_1, C_2, \ldots \}$ of $\R^d$ is a \emph{polygonal tiling} for $d=2$, or a \emph{polyhedral tiling} for $d=3$, if it is normal, each cell $C_i\in \mathcal{M}$ is a convex polygon/polyhedron, and the intersection of any two distinct cells, if not empty, is a common vertex, edge, or face (in the polyhedral case).    
\end{definition}

In these tilings, the nodes are defined as the vertices of the cells. The local structure at a node $V$ is captured by the notion of the vertex figure of $V$; see Figure~\ref{fig_vertexfigure}.

\begin{definition} \label{def:csucsfigura}
    Let $V$ be a node of a polyhedral tiling $\M$. The \emph{vertex figure} $\nu(V)$ of $V$ is 
    \[
    \nu(V) := \{V\} \;\cup\;
    \bigl\{ \text{edges of $\M$ emanating from $V$} \bigr\}
    \cup\;
    \bigl\{ \text{faces of $\M$ containing $V$} \bigr\}.
     \]
\end{definition}

Since in Definition~\ref{def:polihedralis} we required polyhedral tilings to be normal, the vertex figure of $V$ contains only finitely many edges and faces. In order to visualize the combinatorial structure of the vertex figure, we introduce another concept: the vertex polyhedron.

\begin{definition} \label{def:csucspolieder}
    Let $V$ be a node of a polyhedral tiling $\M$, and select $\eps >0$ such that $B^3(V, \varepsilon)$ does not intersect any edge or face of $\M$ that is not contained in the vertex figure $\nu(V)$. Then $\nu(V)$ induces a subdivision of $S^2(V,\varepsilon)$ whose vertices arise as intersections of the sphere with the edges of $\nu(V)$, the edges are obtained by taking the intersections with the faces of $\nu(V)$ (hence these are arcs of great circles), and the faces are given by the intersections of the cells of $\M$ (thus these are spherically convex domains). We say that the polyhedron $P(V)$ is a \emph{vertex polyhedron} of $V$ if $P(V)$ is combinatorially equivalent to this subdivision of the sphere $S^2(V,\varepsilon)$.
\end{definition}

We remark that the induced spherical subdivision is also called the spherical vertex figure. Furthermore, note that the vertex polyhedron is defined only up to combinatorial equivalence.  The notion  is illustrated in Figure~\ref{fig:csucspolieder}. 

\begin{figure} [h]
\centering
\begin{subfigure}{.45\textwidth}
  \centering
  \includegraphics[width=0.7\linewidth]{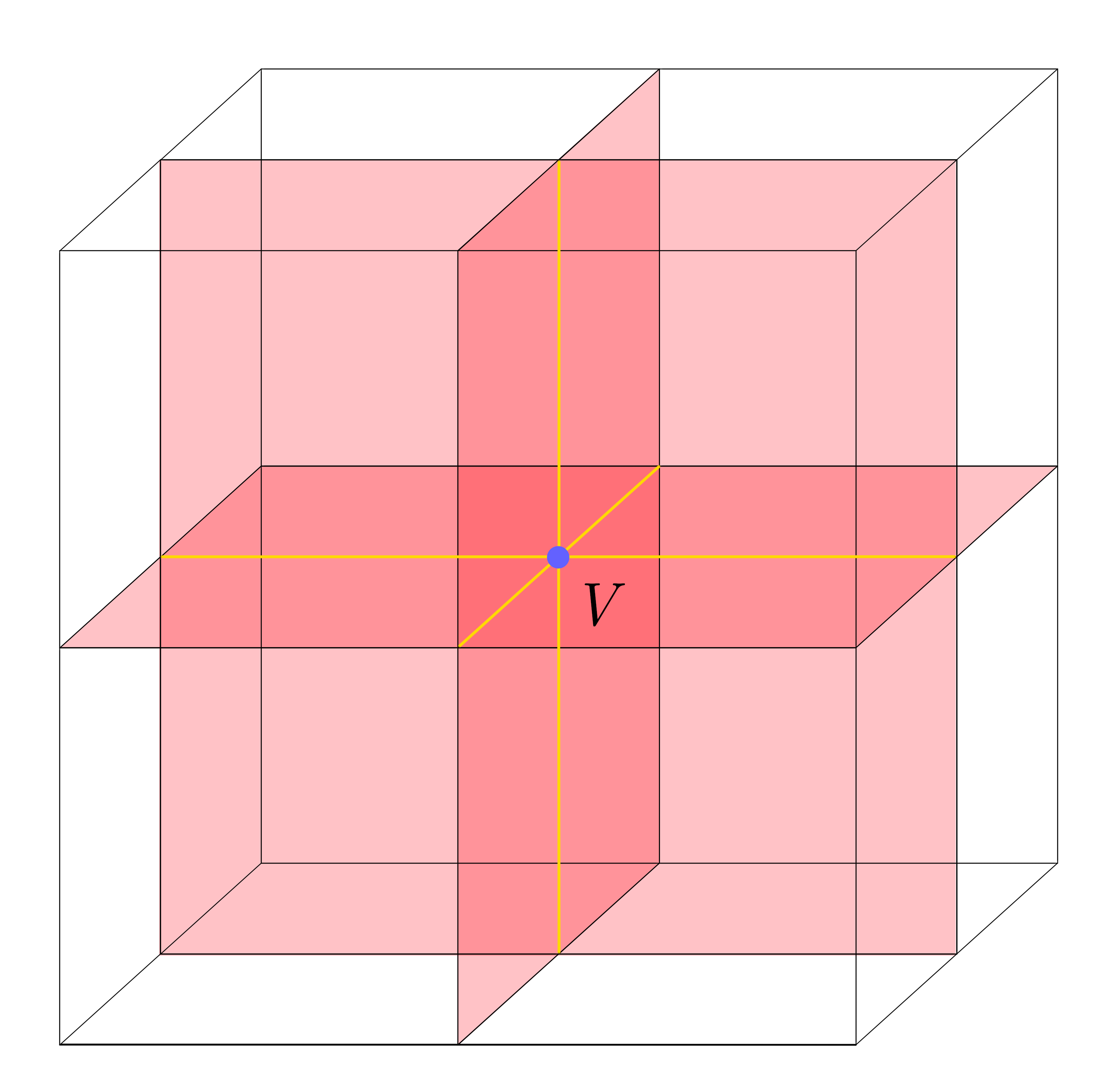} 
  \caption{Vertex figure}
  \label{fig_vertexfigure}
\end{subfigure}
\begin{subfigure}{.45\textwidth}
  \centering
  \includegraphics[width=0.7\linewidth]{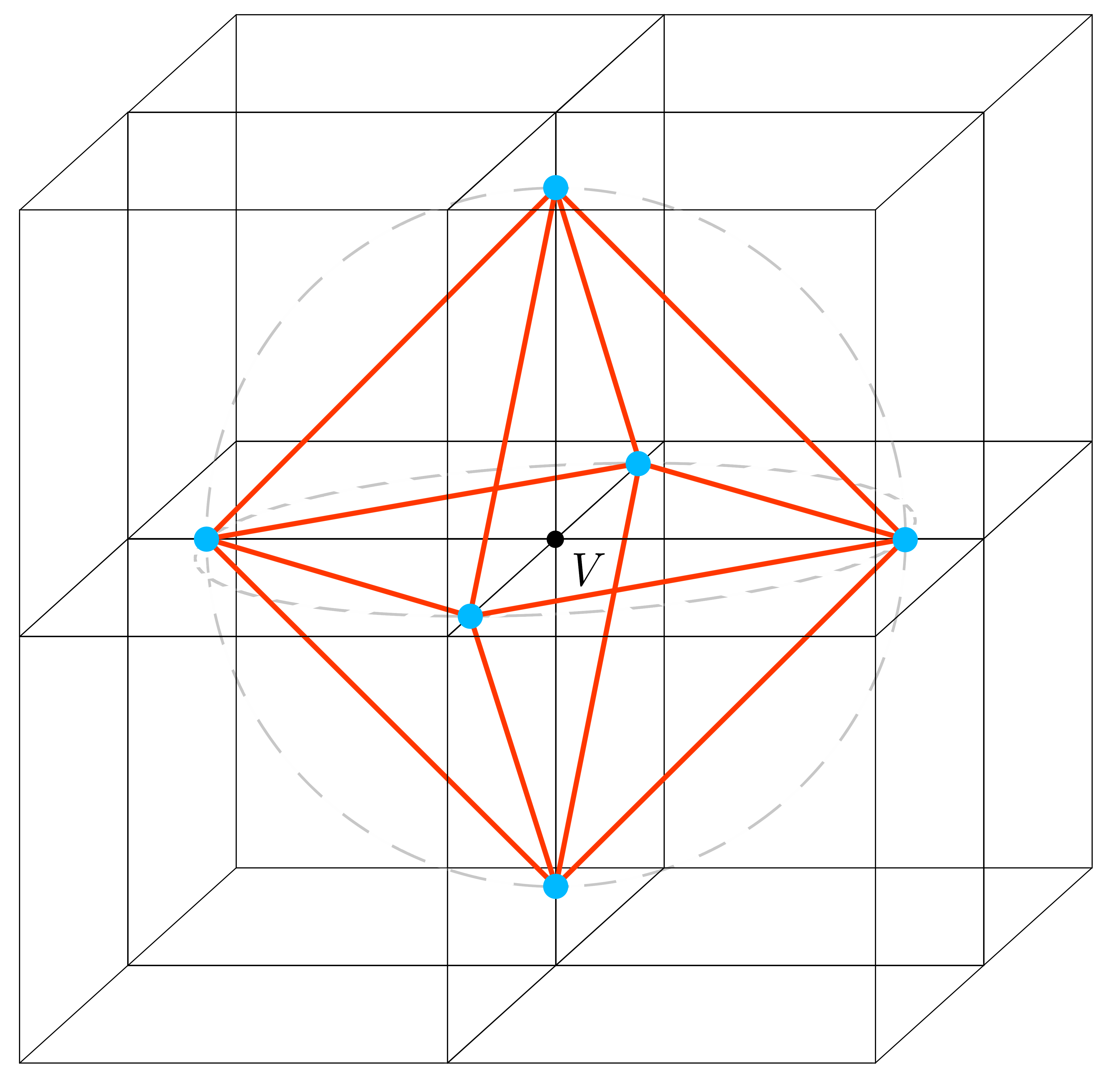}
  \caption{Vertex polyhedron}
    \label{fig:csucspolieder}
\end{subfigure}
    \caption{The vertex figure and vertex polyhedron of a node in a cube tiling}    
\end{figure}

Next, adopting the notions from~\cite{DOM2024}, we introduce polygonic and polyhedric tilings. Let $\M$ be a plane or space tiling. We define two functions associated with $\M$. For a given point $p \in \R^d$, let $C(p)$ denote the number of cells in $\M$ containing $p$, and let $D(p)$ denote the maximal dimension $q$ such that $C$ is constant in some topological $q$-ball containing $p$. Then, we define the \emph{type} of  point $p$ as the pair $(C(p), D(p))$. 

We start with the planar notion.

\begin{definition} \label{def:poligonikus}
    A normal tiling $\M$ of $\R^2$ is called a \emph{polygonic tiling} if it satisfies the following properties:

    \begin{enumerate}
    \item The intersection of any two distinct cells of $\mathcal{M}$ is either empty, a single point, or a  curve.

    \item The type of every point $p \in \mathbb{R}^2$ is one of the following:
    \begin{itemize}
        \item $(1,2)$ -- the topological closure of a maximal connected set of such points is called a \emph{cell} of the tiling;

        \item $(2,1)$ -- the topological closure of a maximal connected component formed by such points is called an \emph{edge} of the tiling;

           \item $(n,0)$ with some $n \ge 3$ -- these points are called the \emph{nodes} of the tiling, or the \emph{vertices} of the cells containing them.
    \end{itemize}
    \item Each edge is a smooth curve connecting exactly two nodes.
    \end{enumerate}
\end{definition}
\noindent
Condition (2) expresses that each edge belongs to exactly two cells, and  each node is contained in at least three cells. 

We continue with the $3$-dimensional analog.

\begin{definition}[\cite{DOM2024}] \label{def:polihedrikus}
    A normal tiling $\M$ of $\R^3$ is called a \emph{polyhedric tiling} if it satisfies the following properties:

    \begin{enumerate}
    \item The intersection of any two distinct cells of $\mathcal{M}$ is either empty, or it is a closed topological ball of dimension at most 2.

    \item The type of every point $p \in \mathbb{R}^3$ is one of the following:
    \begin{itemize}
        \item $(1,3)$ -- the topological closure of the  maximal connected sets of such points  form the \emph{cells} of the tiling;

        \item $(2,2)$ -- the topological closure of the maximal connected components formed by such points form the  \emph{faces} of the cells;

        \item $(m,1)$ with some $m \ge 3$ -- the topological closure of the maximal connected sets of points of this type form the \emph{edges} of the tiling; 

        \item $(n,0)$ with some $n \ge 4$ -- these points are called the \emph{nodes} of the tiling, or the \emph{vertices} of the cells containing them.
    \end{itemize}

    \item The relative interiors of the faces are smooth surfaces.
    \item Each edge is a smooth curve connecting exactly two nodes.
    \end{enumerate}
\end{definition}

In common terms,  condition (2) requires that each face belongs to exactly two cells, each edge is shared by at least three cells, and each vertex is contained in at least four cells. Note that the definition automatically excludes face–edge, face–vertex, and edge-vertex-type cell intersections.
Furthermore, note that the boundary structure of cells is intrinsic to the tiling; that is, the same cell may have different vertices, edges, and faces in different tilings.

The notion of vertex figure introduced for polyhedral tilings in Definition~\ref{def:csucsfigura} extends analogously to polyhedric tilings: given a node $V$, its vertex figure $\nu(V)$ is simply the collection consisting of the vertex $V$ and the edges and faces containing $V$. However, the spherical subdivision induced by $\nu(V)$ considered in Definition~\ref{def:csucspolieder} does not necessarily lead to a simple graph drawn on the sphere, as it may contain parallel edges. We illustrate this phenomenon in Figure~\ref{fig:paplan}: consider the regular cube tiling of $\R^3$, and let $F$ be one of its (square) faces. Modify the tiling by ``inflating'' $F$, that is, replace it by two smooth surfaces (``blankets'') whose intersection is the common, unchanged boundary, while keeping all other faces of the tiling intact. The region of space between these two new faces forms a new cell (the shaded domain in Figure~\ref{fig:paplan1}).
Let $V$ be any vertex of $F$. It is easy to see that the spherical subdivision induced by $\nu(V)$ contains  a region bounded by two parallel edges, see  Figure~\ref{fig:paplan2}. This example shows that in polyhedric tilings, vertex polyhedra may not exist.

\begin{figure}[h]
\centering
\begin{subfigure}{.45\textwidth}
  \centering
  \includegraphics[width=0.8\linewidth]{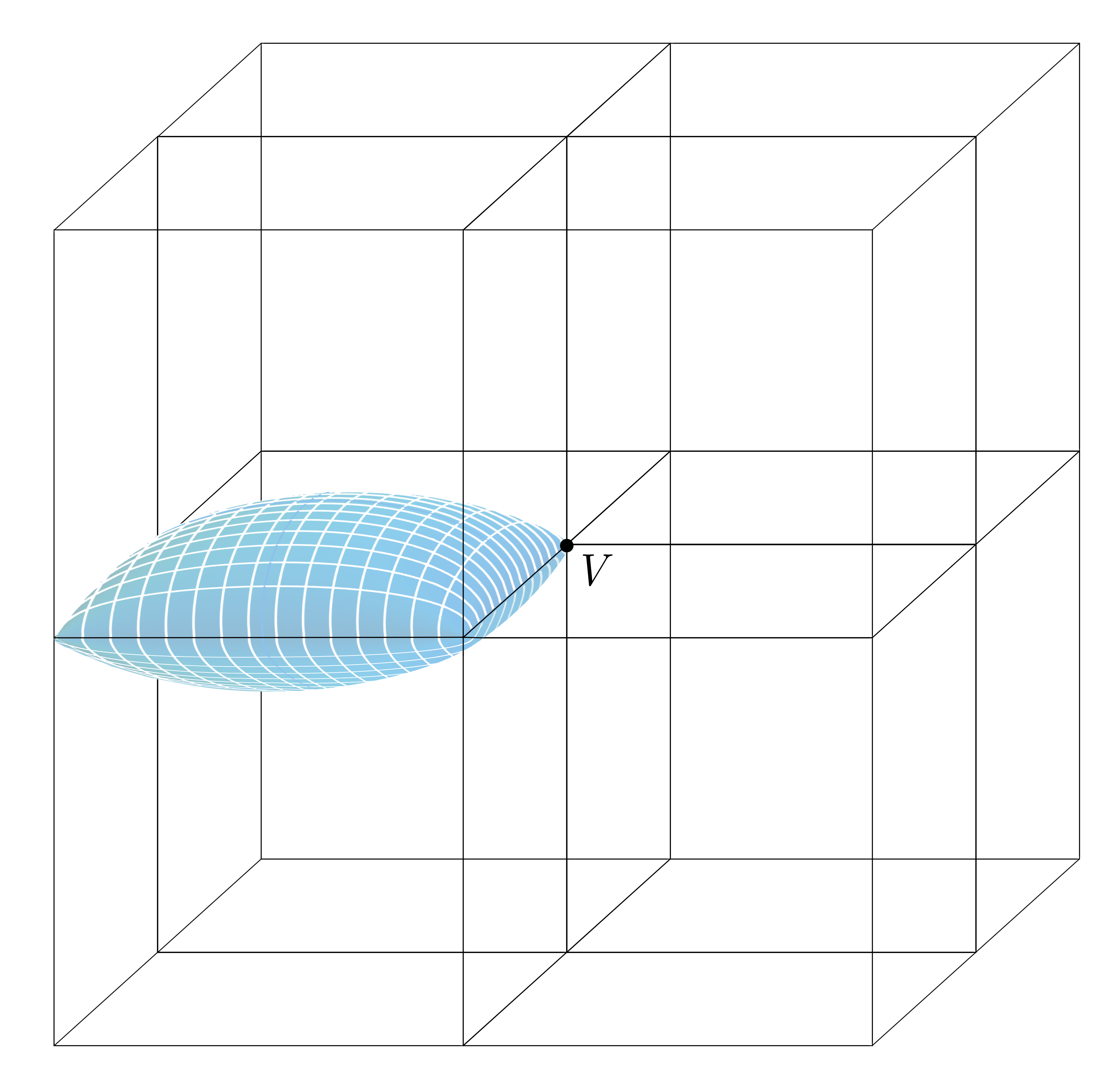}
  \caption{}
  \label{fig:paplan1}
\end{subfigure}
\begin{subfigure}{.45\textwidth}
  \centering
  \includegraphics[width=0.7\linewidth]{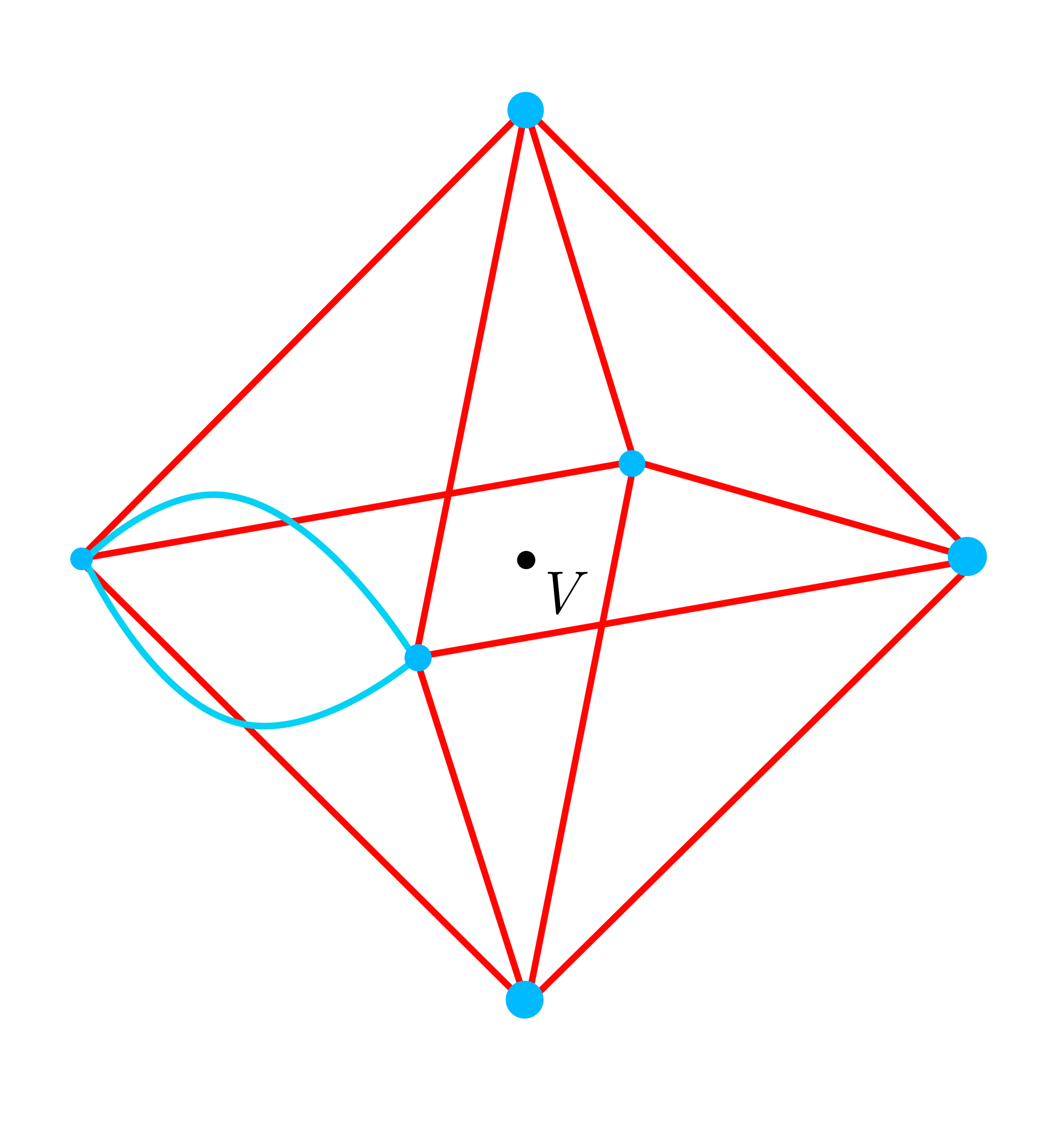}
  \caption{}
  \label{fig:paplan2}
\end{subfigure}
\caption{Local structure of a polyhedric tiling that is not locally polyhedral}
\label{fig:paplan}
\end{figure}

Therefore, it is natural to consider polyhedric tilings in which a neighborhood of each node has the combinatorial structure of a polyhedral tiling.

\begin{definition}\label{def:lokpol}
We call a polyhedric tiling $\M$ of $\R^3$ \emph{locally polyhedral} if, for every node $V$ of $\M$, there exists a positive radius $\eps>0$ such that $B^3(V,\eps)$ does not intersect any edge or face of $\M$ not incident to $V$, and the tiling induced by $\M$ in $B^3(V,\eps)$ is homeomorphic, as a cell complex, to a neighborhood of a node in a polyhedral tiling.
\end{definition}

\noindent
In particular, locally polyhedral tilings have locally finite vertex sets, and the local vertex figure at every node is combinatorially equivalent to a polyhedral vertex figure. By a cell-complex-preserving homeomorphism, we mean a homeomorphism that preserves the combinatorial structure of the decomposition into cells, faces, edges, and nodes.

Next, we introduce the key concepts of the paper.

\begin{definition}\label{def:softcell}
    Given a topological ball $C$ and a point $p \in \bd C$ on its boundary, $p$ is called a {\em spike} of $C$ if there is no smooth curve on $\partial C$ that contains $p$ in its relative interior. Given a tiling $\M$, a cell $C$ of it is called {\em soft} if it does not have spikes and $\M$ is {\em completely soft} if all its cells are soft. 
\end{definition}

We note that, in the three-dimensional case considered throughout most of the paper, the above condition is equivalent to the definition given in~\cite{DOM2024}. However, the two notions differ substantially in the two-dimensional case. We require soft cells to contain no spikes, whereas the authors of~\cite{DOM2024} define a node to be soft if the number of incident cells in which it is a spike cannot be reduced by a smooth local homeomorphism. In particular, the tiling in Figure~\ref{fig:nemlagycella} is soft according to the definition in~\cite{DOM2024}, but does not satisfy the criterion of our Definition~\ref{def:softcell}. For this reason, we avoid using the `soft' term for the formulation of Theorem~\ref{thm:planar} below.  We also note that `spikes' appear as `sharp corners' in~\cite{DOM2023, DOM2024}.

\begin{figure}[h]
    \centering
    \includegraphics[width=0.3\linewidth]{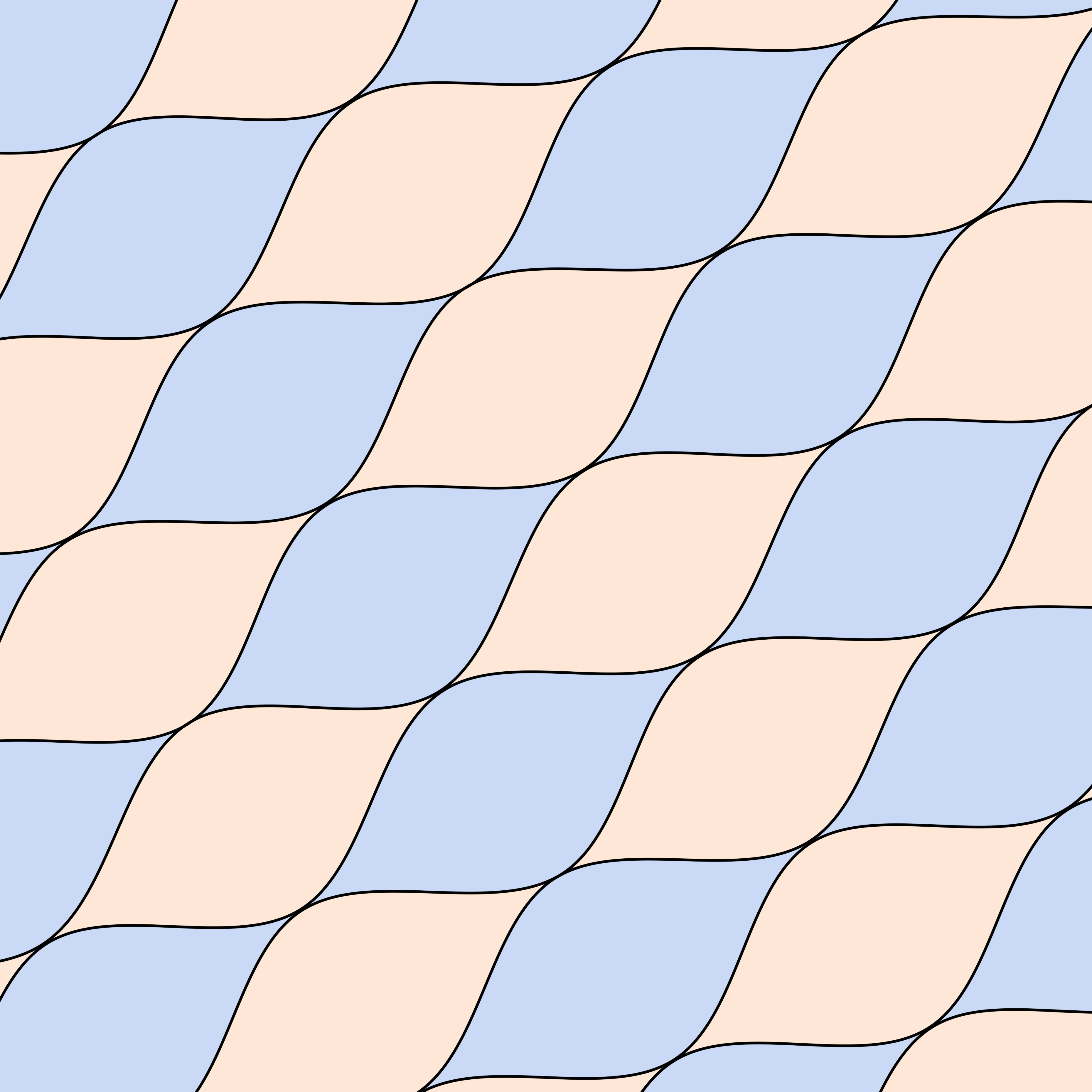}
    \caption{A polygonic tiling that is soft in the sense of \cite{DOM2024}, but not according to Definition~\ref{def:softcell}.}
    \label{fig:nemlagycella}
\end{figure}

As an immediate consequence of Definition~\ref{def:polihedrikus}, the only spikes in a polyhedric tiling can occur at its nodes. 

One of the fundamental questions concerning soft cells is the following: is it possible to tile the plane or space using only cells without spikes? This has recently been studied in \cite{DOM2024,DOM2023}. In the plane, the answer is negative, as the following theorem shows. A normal tiling is called {\em balanced} \cite{grunbaum1987tilings} if the average number of vertices and edges of the cells exists, i.e., calculating the averages in a ball of radius $R$ centered at the origin, these converge to specific values as $R \to \infty$. 

\begin{theorem}[Domokos, G. Horváth, and Regős \cite{DOM2023}]\label{thm:planar}
In a balanced polygonic tiling of the plane, the average number of spikes is at least two per cell. 
\end{theorem}
\noindent
We present a short proof of this result in Section~\ref{sec:plane} that also works for a more general class of smooth surfaces.

The three-dimensional case turns out to be fundamentally different: there is a plethora of completely soft space tilings, as was shown by Domokos, Goriely, G. Horváth and Regős~\cite{DOM2024}. Thus, stronger variants of the question can be studied: Given a polyhedral tiling, does there exist a soft polyhedric tiling that is combinatorially equivalent to it? Furthermore, can the original tiling be transformed ``smoothly'' so that all cells become soft? By virtue of Definition~\ref{def:softcell}, it suffices to find a smooth transformation that, at every node $V$ and cell $C$ that contains it, transforms two edges of $C$ emanating from $V$ into curves whose half-tangents at $V$ complement each other to a line, and whose union is a smooth curve. This method was introduced in \cite{DOM2024}, referred to as the `edge-bending algorithm'. The process is described precisely with the help of the following definitions.

\begin{definition}\label{def:Vlagyitas}
    Let $V$ be a node of a polyhedric tiling $\mathcal{M}$. We say that $V$ can be \emph{completely softened} 
    if there exist a homeomorphism $\Phi \colon \R^3 \to \R^3$, an open neighborhood $K(V)$ of $V$, and a line $t_V$ containing $V$, such that the following conditions are satisfied:
\begin{enumerate}
        \item[(a)] $\Phi(x)=x$ for all $x \notin K(V)$ and $\Phi(V)=V$;
        \item[(b)] the restrictions of $\Phi$ to the relative interiors of the edges and faces of the vertex figure $\nu(V)$ are smooth diffeomorphisms;
        \item[(c)] for every edge $\overline{Vv}$ of the vertex figure $\nu(V)$, the image $\Phi(\overline{Vv})$ is a smooth curve whose half-tangent at $V$ lies on the line $t_V$;  
        \item[(d)] for every cell $C$ of $\M$ containing $V$, there exist vertices $v$ and $w$ of $C$ adjacent to $V$ such that the union of the half-tangents of the curves $\Phi(\overline{Vv})$ and $\Phi(\overline{Vw})$ at $V$ is the complete line $t_V$, and $\Phi(\overline{Vv}) \cup \Phi(\overline{Vw})$ is a smooth curve.
    \end{enumerate}
\end{definition}

The neighborhood $K(V)$ appearing in the definition is called the \emph{bending neighborhood} of the node $V$, while the line $t_V$ is the \emph{axis} of $\Phi$.

\begin{definition}\label{def:Mlagyitas}
    A polyhedric tiling $\M$ can be  \emph{completely softened} if each of its nodes can be softened using pairwise disjoint bending neighborhoods.
\end{definition}

Clearly, if a polyhedral tiling can be completely softened, then there exists a soft polyhedric tiling that is combinatorially equivalent to it, thus yielding an answer to the weaker question as well.

Our main focus is the following question posed by Domokos, Goriely, G. Horváth, and Regős~\cite{DOM2024}: Given a polyhedral tiling of $\R^3$, does there exist an edge-bending algorithm that yields a complete softening?

\begin{conjecture}[{\cite[Conjecture 1]{DOM2024}}]\label{sejtes}
Every polyhedral tiling can be completely softened.
\end{conjecture}

In \cite{DOM2024}, the authors proved that Conjecture~\ref{sejtes} holds for a wide class of polyhedral tilings. 

\begin{theorem}[\cite{DOM2024}]\label{thm:DOM24}
    Assume that, for every node $V$ of the polyhedral tiling $\M$, the vertices of the associated vertex polyhedron $P(V)$ can be partitioned into two classes such that every face contains vertices from both classes and at least one of the two classes is edge-connected. Then $\M$ can be completely softened.
\end{theorem}

The connectedness requirement reflects the inductive nature of the proof: edges of a given vertex figure are bent one at a time, according to a connected ordering of the vertices of the induced vertex subgraph. 

Instead of finding a partition that satisfies the above conditions, the authors also formulated an alternative criterion that uses the concept of the \emph{dual edge graph} of a polyhedron~$P$. This is the planar dual of the edge graph of $P$, whose vertices correspond to the faces of~$P$, and two vertices are connected by an edge if and only if the corresponding faces share a common edge in $P$. In \cite{DOM2024}, it is proved that if at every node $V$ of a polyhedral tiling, the dual edge graph of $P(V)$ contains a Hamiltonian circuit, then the combinatorial property required in Theorem~\ref{thm:DOM24} holds. Consequently, the tiling can be completely softened. However, this latter property is not always satisfied: the smallest example is given by taking the dual of a vertex polyhedron $P(V)$ to be the Herschel graph. In this case, $P(V)$ is an $11$-faced polyhedron that can be realized as the convex hull of the midpoints of the edges of a triangular prism, whose dual edge graph contains no Hamiltonian circuit. 

Nevertheless, the authors of \cite{DOM2024} conjectured that every polyhedron admits a vertex partition that satisfies the conditions of Theorem~\ref{thm:DOM24}. This interesting conjecture, which would yield a proof of Conjecture~\ref{sejtes} via their edge-bending algorithm, remains open. 

\medskip

The main result of the present article establishes the following generalization of Conjecture~\ref{sejtes}:

\begin{theorem}\label{thm:fotetel}
Every locally polyhedral tiling can be completely softened.
\end{theorem}
\noindent 
Our proof relies on a novel edge-bending algorithm that is described in Section~\ref{sec:space}. 
By bending all edges of a vertex figure simultaneously, our method avoids the connectedness requirement imposed in~\cite{DOM2024} on one of the two color classes of the corresponding vertex polyhedron. 

\medskip

The notion of softness can be generalized by requiring higher-dimensional smooth neighborhoods of the boundary points. More precisely, we may define a boundary  point $p$ of a cell $C$ to be $k$-soft if there exists a smoothly embedded topological $k$-ball in $\bd C$ whose relative interior contains $p$. 
With this terminology, the soft cells considered here are precisely those whose boundary points are at least $1$-soft. This convention should not be confused with the codimension-based notion of $k$-softness recently introduced by Domokos~\cite{DOM2026}. That framework provides a more systematic treatment of the smooth submanifolds of different dimensions occurring on the boundary of higher-dimensional shapes and cells in tilings.

\section{Plane tiles have at least two spikes on average}\label{sec:plane}

In this section, we show that the cells of any balanced, polygonic plane tiling have at least two non-smooth vertices on average.

\begin{lemma}\label{lemma:normaltiling}
Suppose that $\M$ is a normal polygonic tiling, all of whose cells have inradius at least $r$ and circumradius at most $R$. Then the number of vertices of any cell is at most
\[
\frac{9 R^2}{r^2}.
\]
\end{lemma}
\begin{proof}
Take any cell $C$ of $\M$. Note that since the nonempty intersection of distinct cells is required to be a vertex or a single curve,  the number of vertices of $C$ equals the number of its neighboring cells, i.e., the cells that share a common edge with $C$.  
Since $C$ fits in a circle of radius $R$, all of its neighboring cells fit in a circle of radius $3 R$ with the same center.  Since the area of each such cell is at least $ r^2 \pi$, the bound readily follows. 
\end{proof}

Note that the same proof also yields that the maximal number of cells containing any given point is bounded above by $ 9 R^2 / r^2$.

\begin{proof}[Proof of Theorem~\ref{thm:planar}]
Let $\M$ be a balanced polygonic tiling. By Definition~\ref{def:poligonikus}, $\M$ is also normal, i.e., the cells are uniformly bounded: each of them contains a circle of radius $r$ and  is contained in a circle of radius $R$.

Consider the circular disc $B^2(o,\rho-2R)$ of radius $\rho- 2R$ centered at the origin $o$. All cells of $\M$ that intersect this disc must be contained in $B^2(o,\rho)$, since the diameter of any cell is at most $2R$.  By local finiteness and the planar cell structure of $\M$, these cells are contained in a finite collection $\M_\rho$ of cells whose union $U_\rho$ is a topological closed disc satisfying
\[
B^2(o,\rho-2R)
\subset
U_\rho
\subset
B^2(o,\rho),
\]
see Figure~\ref{fig:2dbiz}.

\begin{figure}[h]
    \centering
    \includegraphics[width=0.44\linewidth]{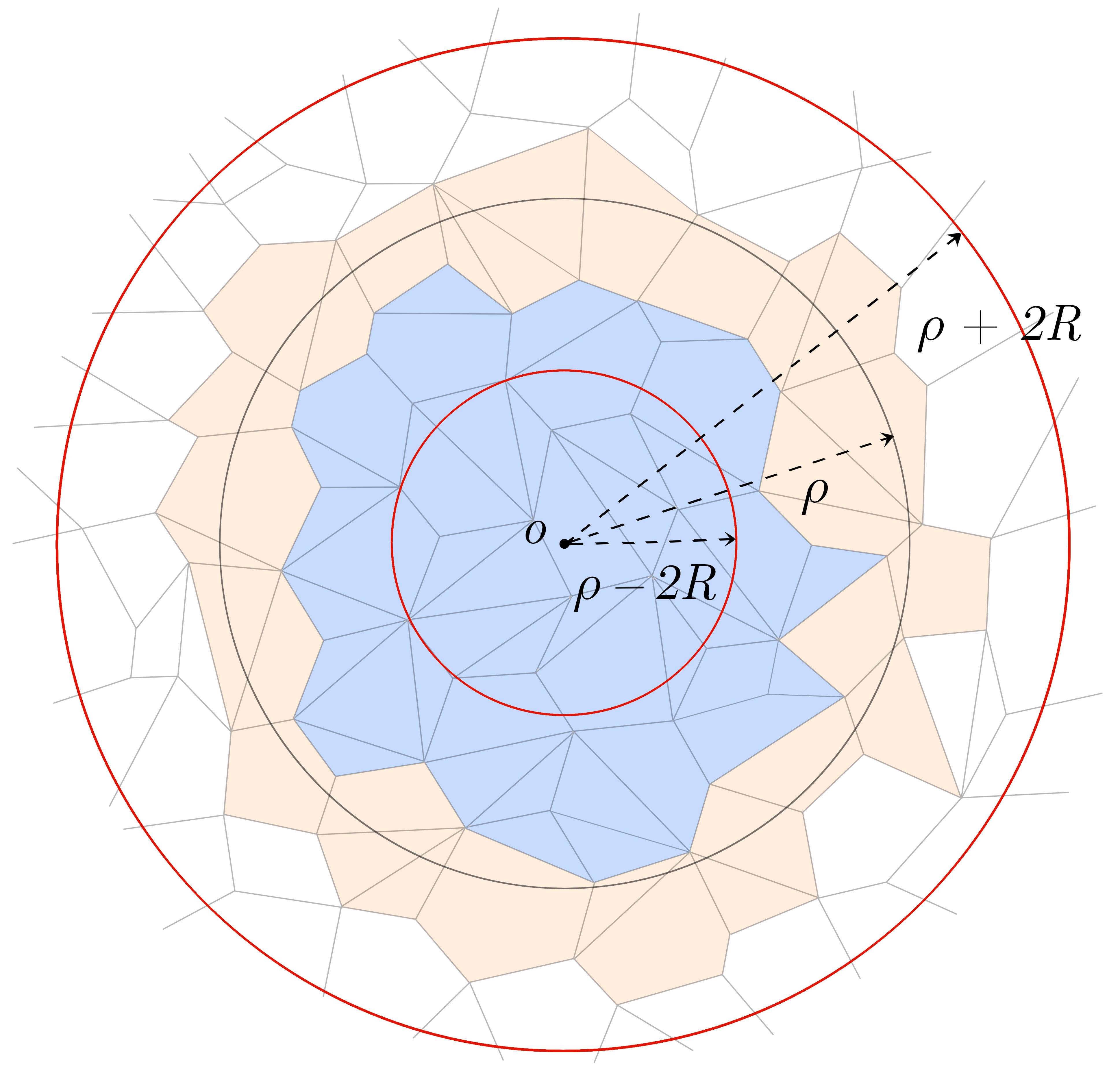}
    \caption{Cells of $\M_\rho$ in blue, cells of $\mathcal{N}_\rho$ in orange contained in the ring between the red circles of radii $\rho -  2 R$ and $\rho + 2 R$ centered at $o$.}
    \label{fig:2dbiz}
\end{figure}

Suppose that $\M_\rho$ consists of $k= k_\rho$ cells -- this number converges to $\infty$ as $\rho \to \infty$.  Denote the cells of $\M_\rho$ by $C_1, \ldots, C_k$. 
Let $G_\rho$ be the skeleton graph of $\M_\rho$, that is, the graph formed by the nodes and edges belonging to the cells of $\M_\rho$. Clearly, $G_\rho$ is a finite connected planar graph  whose faces are precisely the $k$ cells of $\M_\rho$ and its outer face $C_0$.

For each $i = 0, 1, \ldots, k$, denote by  $s_i$  the number of spikes and by $m_i$ the number of smooth vertices on the boundary of $C_i$. Clearly, the number of edges along the boundary of $C_i$ is equal to the total number of vertices therein, that is, $s_i + m_i$.

Note that any node can be a smooth vertex of at most two cells. Therefore, if $n$ denotes the total number of vertices of $G_\rho$, we have
\begin{equation}\label{eq:2n}
\sum_{i=0}^k m_i \leq 2 n.
\end{equation}
On the other hand, if $e$ denotes the number of edges of $G_\rho$, then 
\begin{equation}\label{eq:2e}
\sum_{i=0}^k (s_i + m_i) = 2 e
\end{equation}
because every edge is shared by exactly two faces of $G_\rho$.

Since the number of faces of $G_\rho$ is $k + 1$, Euler's formula implies that 
\[
n - e + (k+1) = 2.
\]
Therefore, 
\begin{align*}
2 &= 2n - 2e + 2k\\
&\geq   \sum_{i=0}^k m_i - \sum_{i=0}^k (s_i + m_i) + 2 k \\ 
&= 2 k - \sum_{i=0}^k s_i 
\end{align*}
\noindent
by \eqref{eq:2n} and \eqref{eq:2e}. This leads to 
\begin{equation}\label{eq:s_average}
    \frac{\sum_{i=1}^k s_i }{k} \geq 2 - \frac{s_0 + 2}{k}.
\end{equation}

The quantity on the left hand side above is precisely the average number of spikes of the cells in  $\M_\rho$. Thus, it suffices to show that $\frac{s_0 }{k}$ tends to $0$ as $\rho \to \infty$.

First, note that according to the normality assumption, each cell in $\M_\rho$ has an area at most that of a circular disk of radius $R$. Since the cells in $\M_\rho$ completely cover $B^2(o,\rho - 2R)$,  area comparison gives that 
\begin{equation}\label{eq:kest}
k \geq \frac{(\rho - 2 R)^2}{R^2} \geq \frac 1 {2 R^2} \cdot \rho^2
\end{equation}
if $\rho$ is sufficiently large. 

Next, we estimate $s_0$, the number of spikes along the boundary of $\M_\rho$. Let $\mathcal{N}_\rho$ be the set of cells in $\M \setminus \M_\rho$ that intersect (i.e., share a point with) a cell in $\M_\rho$, see Figure~\ref{fig:2dbiz}.  Note that by the diameter bound, each cell in $\mathcal{N}_\rho$ is contained in the ring between the circles of radii $\rho -  2 R$ and $\rho + 2 R$ centered at $o$, which has area $8 \pi R \cdot \rho$. Since each cell in $\mathcal{N}_\rho$ has area at least $\pi r^2$, their number is at most  $(8 R/r^2)\cdot \rho$.

Clearly, $s_0$ is at most the total number of vertices of cells in $\mathcal{N}_\rho$. Therefore, combining Lemma~\ref{lemma:normaltiling} with the above bound, we derive that 
\[ 
s_0 \leq \frac{9 R^2}{r^2} \cdot \frac{8 R}{r^2} \cdot \rho = \frac{72 R^3}{r^4} \cdot \rho.
\]
Consequently, by \eqref{eq:kest},
\[
\frac{s_0}{k} \leq \frac{144 R^5}{r^4} \cdot \frac 1 \rho
\]
which converges to 0 as $\rho \to \infty$.
\end{proof}

\begin{remark}
We note that the same proof implies that the average number of vertices with interior angle at most $2 \pi /3$ is at least two per cell. This fact was pointed out by an anonymous referee. 
\end{remark}

\begin{remark}
Following the suggestion of an anonymous referee, we also point out that Theorem~\ref{thm:planar} extends to smooth surfaces $\mathcal{S}$ that are either compact or admit a suitable notion of averaging, as in the case of \cite[Theorem 1]{DOM2023} -- see Remarks 2 and 3 therein. The only difference is that  \eqref{eq:s_average} transforms to 
\[
\frac{\sum_{i=1}^k s_i }{k} \geq 2 - \frac{s_0 + 2 \chi(\mathcal{S})}{k},
\]
where $ \chi(\mathcal{S})$ is the Euler characteristic of $\mathcal{S}$. 
If $\mathcal{S}$ is compact, then this provides an explicit lower bound. For noncompact surfaces with the required properties, the contribution of the term $\frac{2 \chi(\mathcal{S})}{k}$ converges to 0, hence we recover the estimate of Theorem~\ref{thm:planar}.
\end{remark}

\section{Softening locally polyhedral tilings}\label{sec:space}

The proof of Theorem~\ref{thm:fotetel} consists of two main steps. First, we show that a specific combinatorial condition holds for polyhedral graphs. Then, based on this property, we construct a novel edge-bending algorithm.

Given a polyhedral graph $G$, a {\em triangulation} of it is an augmented graph $G_\Delta$ that is obtained from $G$ by adding edges (while keeping the same vertex set) so that each face of the resulting graph is a triangle. This can be achieved by subdividing each face of $G$ (including the outer face) into triangles using pairwise non-crossing diagonals (e.g., using all diagonals starting from a given vertex of the face).

\begin{lemma}\label{lemma:policolor}
    Let $G$ be a polyhedral graph. Then the vertices of $G$ can be colored with two colors so that there are no monochromatic faces, i.e., every face contains vertices of both color classes.
\end{lemma}

We provide two proofs of the above statement.

\begin{proof}[First proof of Lemma~\ref{lemma:policolor}]
We will use a special case of Thomassen's result on list colorings~\cite{THO2008} stating that, given a planar graph, one can assign one of two colors to each vertex so that no triangle is monochromatic.

Let now $G_\Delta$ be a triangulation of $G$. Take a coloring of the vertices provided by the aforementioned result. This coloring satisfies the requirement of Lemma~\ref{lemma:policolor}, since if some face of $G$ was monochromatic, then all triangles of $G_\Delta$ arising from that face would also need to be monochromatic. Therefore, every face of $G$ necessarily contains vertices from both color classes.
\end{proof}

We note that in 2021, Karpiński and Piecuch \cite{KARPINSKI2021} gave a linear-time algorithm to construct a two-coloring of simple, undirected planar graphs that contain no monochromatic triangles. 

\begin{proof}[Second proof of Lemma~\ref{lemma:policolor}]
As before, let $G_\Delta$ be a triangulation of $G$. By the Four Color Theorem \cite{APPEL77p1, APPEL77p2}, the graph $G_\Delta$ admits a proper vertex coloring with four colors, so that adjacent vertices receive different colors. In particular, the vertices of every triangle receive three distinct colors. Let the four colors be $\{1,2,3,4\}$, and set $I = \{1,2\}$, $II = \{3,4\}$. Define a new coloring of $G_\Delta$ using the colors $I,II$ by coloring each vertex according to the class to which its original color belongs. Clearly, in this new coloring, no monochromatic triangle can occur, hence this is also a coloring of $G$ that meets the requirements. 
\end{proof}

The edge-bending algorithm will transform straight line edges to curves that satisfy certain properties. To that end, we introduce the function $\varphi \colon [0,1] \to \R$ defined as
\begin{equation}\label{eq:phi}
\varphi(x)= x^{1/3}(1-x)^3,
\end{equation}
see Figure~\ref{fig:phi}.

\begin{figure}[h]
    \centering
    \includegraphics[width=0.5\linewidth]{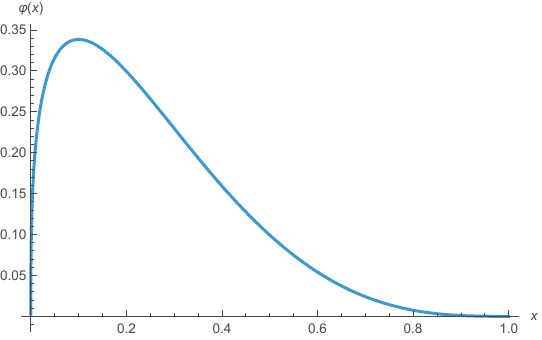}
    \caption{The function $\varphi(x)$}
    \label{fig:phi}
\end{figure}

\begin{lemma} \label{lemma:phi}
The function  $\varphi$ satisfies the following properties:
    \begin{enumerate}
        \item[(i)] $0 \leq \varphi(x) \leq 1$ for every $x \in [0,1]$ and $\varphi(0) = \varphi(1) = 0 $;
        \item[(ii)] $\varphi \in C^2 \big((0,1)\big)$;

        \item[(iii)]
    \[
    \lim_{x\to0^+}\varphi'(x)=+\infty
    \qquad\text{and}\qquad
    \lim_{x\to0^+}
    \frac{\varphi''(x)}{(\varphi'(x))^3}
    =0;
    \]
    
    \item[(iv)]
    \[
    \lim_{x\to1^-}\varphi'(x)=0
    \qquad\text{and}\qquad
    \lim_{x\to1^-}\varphi''(x)=0.
    \]
    \end{enumerate}
\end{lemma}

\begin{proof} Direct differentiation gives 
\[
\varphi'(x) = \frac{(1-x)^2(1-10x)}{3x^{2/3}} 
\]
and
\[
\varphi''(x) = -\frac{2(1-x)(1+7x-35x^2)}{9x^{5/3}}. 
\]
Properties~(i), (ii), the first part of (iii), and (iv) follow immediately. Moreover, 
\[
\frac{\varphi''(x)}{(\varphi'(x))^3} = -\frac{ 6x^{1/3}(1+7x-35x^2) }{ (1-x)^5(1-10x)^3 }, 
\]
which tends to $0$ as $x\to0^+$. This proves~(iii). \end{proof}

The usefulness of properties (iii) and (iv) is illustrated by the following statement. We extend  $\varphi$ to $[0, \infty)$ by defining $\varphi(x)=0$ for $x>1$. Because of property (iv), the resulting function is smooth (i.e., $C^2$) on $(0, \infty)$.

\begin{lemma} \label{lemma:smoothjoin}
Assume that $f:[-2,0] \to  \R^3$ and  $g: [0,2] \to \R^3$ are parametrizations of smooth curves with a single intersection point $f(0) = g(0)$. Let $\zz \in \R^3$ be a nonzero vector that is not parallel to any of the tangents or half-tangents of $f$ and $g$. Assume, furthermore, that no chord joining two distinct points of $f([-2,0])\cup g([0,2])$ is parallel to $\zz$. 
Define
\[
F(x) = \begin{cases}
 f(x) - \varphi(-x)\cdot \zz &\text{ for } x<0 \\
 g(x) + \varphi(x)\cdot \zz &\text{ for } x \geq 0
 \end{cases}
\]
for $x \in [-2,2]$. Then $F([-2,2])$ is a smooth curve  whose tangent at $F(0)$ is parallel to $\zz$.    
\end{lemma}

\begin{proof}
Global injectivity of $F$ follows immediately: if $F(x) = F(y)$ for some $x, y \in [-2,2]$, then the corresponding original points must differ by a scalar multiple of $\zz$, which contradicts the choice of $\zz$ or the global injectivity of $f$ and $g$.

For $x >0$, we have 
\[
F'(x) = g'(x) + \varphi'(x) \zz.
\]
This vector is nonzero, since otherwise $g'(x)$ would be parallel to
$\zz$. An  analogous argument applies to $x <0$, leading to the regularity of the parametrization for $x \neq 0$. The $C^2$ property on $(-2,0)$ and $(0,2)$ follows from the differentiability properties of $f, g$ and $\varphi$.  Moreover, $F$ agrees with $f$ on $[-2,-1]$ and with $g$ on
$[1,2]$, which shows the endpoint regularity. 

It remains to consider the join at $x_0 = 0$. 
After translating, we may assume that $f(0) = g(0) = F(0) = o$. 
Since $\varphi'(x)>0$ for $0<x<1/10$, the restriction of $\varphi$
to a sufficiently small interval $[0,\delta]$ is strictly increasing.
Introduce the signed parameter
\[
t = \begin{cases}
- \varphi(-x) &\text{ for } - \delta< x<0 \\
\varphi(x) &\text{ for } 0 \leq  x \leq \delta.
\end{cases}
\]
Let $I= [-\delta, \delta]$ and $J=[-\varphi(\delta), \varphi(\delta)]$. Define $q$ to be the local inverse of $\varphi$:
\[
q = \bigl(\varphi|_{[0,\delta]} \bigr)^{-1}.
\]
For $t>0$, the inverse-function formulas give
\[
q'(t) = \frac{1}{\varphi'(q(t))}
\]
and
\[
q''(t) = -\frac{\varphi''(q(t))}{\bigl(\varphi'(q(t))\bigr)^3}.
\]
By property~(iii) of Lemma~\ref{lemma:phi},
\[
\lim_{t\to0^+}q'(t)=0
\qquad\text{and}\qquad
\lim_{t\to0^+}q''(t)=0.
\]
Hence, $q$ extends to a $C^2$ function at $0$ by setting
\[
q(0)=q'(0)=q''(0)=0.
\]
We can describe $F(I)$ in terms of $t$ by the mapping $\Gamma: J \to F(I)$ defined as  
\[
\Gamma(t)=
\begin{cases}
f\bigl(-q(-t)\bigr)+t \cdot \zz&\text{ for }t<0,\\
g\bigl(q(t)\bigr)+t \cdot \zz &\text{ for }t\geq0.
\end{cases}
\]
For $t>0$, we have
\[
\Gamma'(t) = g'\bigl(q(t)\bigr) \bigl(q(t)\bigr)'+\zz
\]
and
\[
\Gamma''(t) = g''\bigl(q(t)\bigr) \bigl((q(t))'\bigr)^2+g'\bigl(q(t)\bigr) \bigl(q(t)\bigr)'' .
\]
Consequently, 
\[
\lim_{t\to0^+}\Gamma'(t)=\zz
\qquad\text{and}\qquad
\lim_{t\to0^+}\Gamma''(t)=0.
\]
For $t<0$, analogous calculations establish the limits
\[
\lim_{t\to0^-}\Gamma'(t)=\zz
\qquad\text{and}\qquad
\lim_{t\to0^-}\Gamma''(t)=0.
\]
Thus, $\Gamma$ is a $C^2$ parametrization of $F(I)$ across $t=0$, with
\[
\Gamma'(0)=\zz\neq0
\qquad\text{and}\qquad
\Gamma''(0)=0.
\]
It is therefore regular at the joining point, and its tangent there
is parallel to $\zz$. This shows that $F([-2, 2])$ is indeed locally smooth everywhere. Since these local regular parametrizations are compatible and $F$ is injective, this yields a global $C^2$ parametrization of $F([-2,2])$,  completing the proof.
\end{proof}

After these technical preparations, we describe the edge-bending algorithm that produces  complete softenings of locally polyhedral tilings. We first treat the polyhedral case and then explain the modifications required for locally polyhedral tilings.

\begin{proof}[Proof of Theorem~\ref{thm:fotetel} for polyhedral tilings.]
We show that each node $V$ of a polyhedral tiling  $\M$  can be softened by a local deformation $\Phi_V$ in such a way that the bending neighborhoods associated with distinct nodes are pairwise disjoint. 
This guarantees that the tiling $\M$ can be completely softened, cf. Definition~\ref{def:Mlagyitas}.

We start by choosing an open ball around each node of $\M$ such that the selected balls are pairwise disjoint, and the ball assigned to a node  $V$ does not meet any face or edge of $\M$ not incident to $V$. These balls will contain the bending neighborhood of the relevant nodes, thus ensuring that the local homeomorphisms $\Phi_V$ can be combined to obtain a global homeomorphism $\Phi$.

For a node $V$ of the tiling, let $B^3(V,r)$ denote the ball assigned to $V$, with boundary $S^2(V,r)$. The intersection of the vertex figure $\nu(V)$ with $S^2(V,r)$ may be regarded as a graph $G$ embedded in $S^2(V,r)$, combinatorially equivalent to the edge graph of a vertex polyhedron of $V$. The vertices and edges of $G$ arise from the intersections of $S^2(V,r)$ with the edges and faces, respectively, of $\M$ containing $V$. Since $\M$ is a normal tiling, only finitely many such edges and faces contain $V$, and hence $G$ is a finite graph.

Select a line  $t_V$ through $V$ that is not coplanar with the line connecting two distinct vertices of $G$, unless that line contains $V$. This implies that $t_V$ does not contain any vertex of $G$ and does not intersect any edge of $G$. From now on, we refer to $t_V$ as the {\em axis}.

We parametrize $\R^3$ with respect to the axis $t_V$ using cylindrical coordinates: place the  origin at $V$ and select an orthonormal basis $\mathbf{x}, \mathbf{y},\mathbf{z}$ so that $\mathbf{z}$ is parallel to $t_V$. Then, for $\eps >0$, $\alpha \in \R$, and $h \in \R$, let the triple $(\varepsilon,\alpha,h)$ describe the point
\[
p(\varepsilon,\alpha,h) = \varepsilon \cos \alpha  \cdot \mathbf{x} + \varepsilon \sin \alpha  \cdot \mathbf{y} + 
h \cdot \mathbf{z}.
\]
Thus, $\varepsilon$ denotes the distance of $p$ from the axis $t_V$, $\alpha \in [0,2\pi)$ is the azimuth angle of $p$ (understood modulo $2 \pi$), and $h \in \R$ is the signed distance from $V$ of the orthogonal projection of $p$ onto $t_V$.

In particular, the cylindrical surface consisting of points at a distance $\varepsilon$ from $t_V$ is given by    
\[
    H_\varepsilon=\{(\varepsilon,\alpha,h):\alpha\in[0,2\pi),\, h\in\mathbb{R}\}.
\]

Assume that $G$ has $n$ vertices $w_1, \ldots, w_n$, where $w_i = (\eps_i, \alpha_i, h_i)$ for $i \in [n]:=\{1, \ldots, n\}$. We may also suppose that 
\[
0 = \alpha_1 < \alpha_{2} < \cdots < \alpha_n < 2 \pi
\]
since no two vertices share the same azimuth angle, due to the choice of $t_V$. Also, define $\alpha_{n+1}:= 2 \pi$; equivalently, assume that the values $\alpha_i$ are interpreted cyclically.

Furthermore, choose $\kappa$ with $0<\kappa<r$ sufficiently small that
\begin{equation}\label{eq:kappadef}
\nu(V) \cap H_\kappa \subset B^3\!\left(V,\frac{r}{4}\right)
\end{equation}
and
\begin{equation}\label{eq:kappa_ball}
\Big(\kappa, \alpha, \frac{3 r}{4} \Big) \in B^3 (V,r)
\end{equation}
for every $\alpha \in [0, 2 \pi)$.
The latter condition is equivalent to requiring that 
\[
\kappa^2+\left(\frac{3r}{4}\right)^2<r^2.
\]
Such a choice of $\kappa$ is possible because, by the choice of the axis, $t_V\cap\nu(V)={V}$, whereas the closure of $t_V\cap B^3(V,r)$ is a segment of length $2r$, and the relevant intersections vary continuously with $\kappa$.

The local homeomorphism $\Phi_V$ is constructed from a family of translations parallel to the axis $t_V$, defined separately on the cylindrical surfaces $H_\varepsilon$ for $0<\varepsilon \leq \kappa$. The magnitude of translation, denoted by  $T(\eps, \alpha) \in \R$ where $\eps \geq 0$ and $\alpha \in [0, 2 \pi)$, depends only on the distance $\varepsilon$ from the axis $t_V$ and the azimuth angle $\alpha$. It is obtained as a smooth approximation of a function $\tau(\eps, \alpha)$ that is piecewise linear in $\alpha$ for each fixed~$\eps$. We  define~$\tau$ as follows. 

\medskip

\begin{itemize}

\item
    First, we describe the translation along the directions corresponding to the edges of the vertex figure. Take a coloring of the vertices of $G$ with two colors $\{ +1, -1 \}$ provided  by Lemma~\ref{lemma:policolor}, so that each face contains vertices from both color classes. For  $i \in [n]$, let $\sigma_i$ be the color (i.e., sign) assigned to the vertex $w_i$, and define $\sigma_{n+1}:= \sigma_1$. For $\varepsilon \in [0,\kappa]$ and $i \in [n+1]$, we define
    \begin{equation}\label{eq:tea}
    \tau(\eps, \alpha_i): = \frac r 4 \cdot \sigma_i \cdot \varphi\left(\frac \eps \kappa \right),
    \end{equation}
    where $\varphi$ is given by \eqref{eq:phi}. Note that the multiplicative constant $r/4$ ensures that the translated images remain in $B^3 \! (V,r/2)$, while the condition that both signs are required for each face ensures the existence of a smooth curve through $V$ in the boundary of each of the corresponding cells.

\item
   Next, for all other angular directions we interpolate linearly between consecutive vertices:
    If $\varepsilon \in [0,\kappa]$ and $\alpha \in (\alpha_i,\alpha_{i+1})$ for some $i \in [n]$
    (remember that $\alpha_1=0$ and $\alpha_{n+1}=2\pi$),
    then there exists a unique $\lambda \in (0,1)$ such that
    \[
    \alpha = (1-\lambda)\cdot\alpha_i + \lambda\cdot \alpha_{i+1}.
    \]
    We then set
    \[
    \tau(\varepsilon,\alpha)
    :=
    (1-\lambda)\cdot\tau(\varepsilon,\alpha_i)
    +
    \lambda\cdot \tau(\varepsilon,\alpha_{i+1}).
    \]

\item 
    Finally, for $\eps > \kappa$, let 
    \[
    \tau(\eps, \alpha):= 0.
    \]
This  guarantees that condition (a) is satisfied.

\end{itemize}
\medskip

\noindent 
The function $\tau$ is illustrated in Figure~\ref{abra:tau} for a set of four evenly distributed edge directions with alternating signs. 

Note that the above properties imply that for every  $0 \leq \eps\leq \kappa$ and every $\alpha \in \R$,
\[
    |\tau(\eps, \alpha)| \leq \frac r 4.
\]

\begin{figure}[H] 
    \centering
    \includegraphics[width=0.7\linewidth]{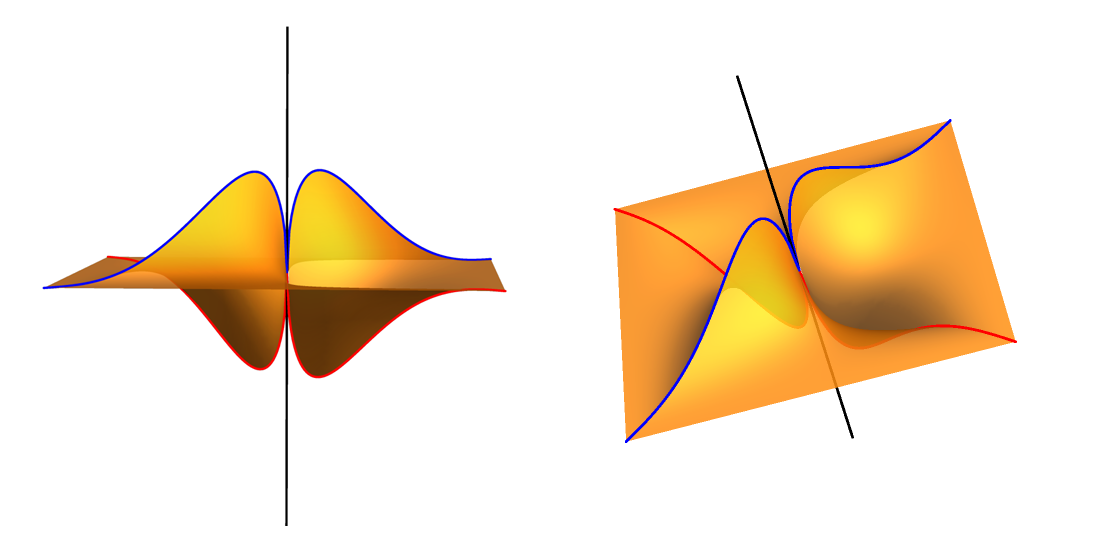}
    \caption{The function $\tau(\varepsilon, \alpha)$ in the case  $n=4$, $\alpha = \{ \frac{\pi}{4}, \frac{3\pi}{4}, \frac{5\pi}{4}, \frac{7\pi}{4} \}$, and $\sigma = \{ +1, -1, +1, -1 \}$}
    \label{abra:tau}
\end{figure}

The map $\tau(\eps, \alpha)$, expressed in the polar coordinates $\eps \geq 0$ and $\alpha \in [0, 2 \pi)$, is continuous everywhere and smooth except when $\eps=0$ or $\alpha=\alpha_i$ for some $i \in [n]$. This follows readily from the properties of the function $\varphi$: in the interior of each region determined by consecutive angles $\alpha_i$ and $\alpha_{i+1}$ this follows from part~(ii) of Lemma~\ref{lemma:phi} and the linear interpolation in the angular variable, while at $\eps = \kappa$, parts (i) and (iv) ensure the smooth join of $\tau$ to the zero function.

We now modify $\tau$ to obtain an interpolating function $T(\eps, \alpha)$  that is smooth everywhere except at the origin, corresponding to $\eps=0$, while preserving its values along the rays $\alpha=\alpha_i$. One explicit construction of $T$, illustrated in Figure~\ref{fig:tauT}, is the following.

Let $\delta >0$ be a small positive number for which the circular arcs  
$[\alpha_i - \delta, \alpha_i +\delta] \mod 2 \pi$  are all disjoint for $ i \in [n]$. Define
\[
    T(\eps, \alpha) = \begin{cases}
\tau(\eps, \alpha) &\text{ if } |\alpha - \alpha_i| > \delta \text{ for every } i \in [n] \\
\tau(\eps, \alpha_i) + \psi\left(\frac{|\alpha - \alpha_i|}{\delta}\right) (\tau(\eps, \alpha) - \tau(\eps, \alpha_i)) &\text{ if } |\alpha - \alpha_i| \leq \delta
    \end{cases}
\]
where $|\cdot|$ denotes the circular distance, and
\[
\psi(x)= x - \frac{\sin(2 \pi x)}{2 \pi}
\]
for $x \in [0,1]$. Note that $\psi(0) = \psi'(0) = \psi''(0) =0$ while $\psi(1)=1$ and  $ \psi'(1) = \psi''(1) =0$, thus $T$ is indeed smooth outside the origin, and it agrees with $\tau$ on the rays corresponding to each $\alpha_i$. Furthermore, it satisfies 
\begin{equation}\label{eq:T}
    |T(\eps, \alpha)| \leq \frac r 4 \cdot \varphi \left( \frac \eps \kappa\right)
\end{equation}
for every $0 \leq \eps\leq \kappa$, $\alpha \in \R$.

\begin{figure} [H]
\centering
\begin{subfigure}{.45\textwidth}
  \centering
  \includegraphics[width=0.7\linewidth]{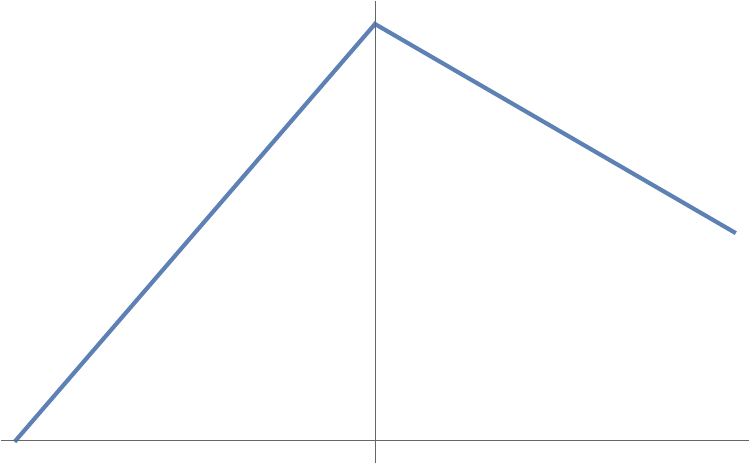} 
  \caption{$\tau(\eps, \alpha)$}
\end{subfigure}
\begin{subfigure}{.45\textwidth}

\centering
  \includegraphics[width=0.7\linewidth]{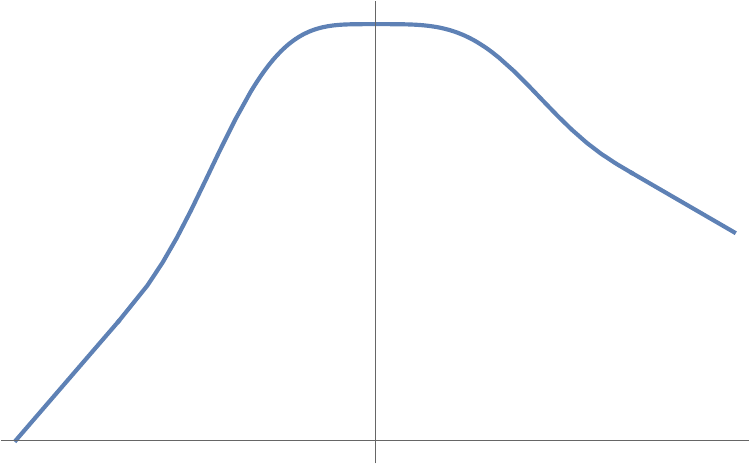}
  \caption{$T(\eps,\alpha)$}
\end{subfigure}
    \caption{The functions $\tau(\eps, \alpha)$ and  $T(\eps, \alpha)$ for fixed $\eps$, around $\alpha = \alpha_i$}    
    \label{fig:tauT}
\end{figure}

To ensure that $\Phi_V$ coincides with the identity outside the ball $B^3(V,r)$, we introduce the cutoff function $\mu: \R \to [0,1]$ defined by
\begin{equation}\label{eq:mudef}
\mu(h) = \begin{cases}
1 &\text{ if } |h| < \dfrac r 4 \\
\dfrac 3 2 - \dfrac 2 r \cdot |h|  &\text{ if }  \dfrac r 4 \leq |h| \leq \dfrac {3r} 4\\
0 &\text{ if } |h| > \dfrac {3r} 4 \\
\end{cases}
\end{equation}
which is a continuous, piecewise linear function on $\R$.

Having established all the necessary ingredients, we define the local homeomorphism  $\Phi_V \colon \R^3 \to \mathbb{R}^3$ in cylindrical coordinates by
\begin{equation} \label{eq:phi0def}
\Phi_V(\eps, \alpha, h)  := \big(\eps, \alpha, h + \mu(h) \cdot T(\eps, \alpha)\big).
\end{equation}
We need to verify that $\Phi_V$ satisfies the conditions of Definition~\ref{def:Vlagyitas}.

First, by the definition of $\mu$ and the bound $|T(\eps,\alpha)|\leq r/4$, the restriction of $\Phi_V$ to every line parallel to $t_V$, corresponding to fixed values of $\eps$ and $\alpha$, is a continuous bijection. Its action on such a line is easy to visualize: it translates the interval $[-r/4,r/4]$ by the amount $T(\eps,\alpha)$, fixes the two half-lines $(-\infty,-3r/4]$ and $[3r/4,\infty)$, and interpolates linearly between the translation and the identity on the two remaining intervals. Consequently, the restriction of  $\Phi_V$ to any line parallel to $t_V$ is a homeomorphism.

Since $\Phi_V$ preserves the coordinates $\eps$ and $\alpha$, these linewise homeomorphisms imply that $\Phi_V$ is bijective. Moreover, the continuity of $T$, the identity $T(0,\alpha)=0$, and the property 
\[
\sup_\alpha |T(\eps, \alpha)| \to 0 \text{ as } \eps \to 0^+
\]
imply that $\Phi_V$ is continuous on $\R^3$, including along the axis $t_V$. To see that its inverse is also continuous, fix $\eps$ and $\alpha$, and write  $c=T(\eps,\alpha)$. The restriction of $\Phi_V$ to the corresponding line parallel to $t_V$ is
\[
f_c(h)=h+\mu(h)c.
\]
This is a piecewise linear function whose slopes by \eqref{eq:mudef} are $1$, $1+\frac{2c}{r}$, or $1-\frac{2c}{r}.$
Since $|c|\leq r/4$, all these slopes are at least $1/2$ and at most $3/2$. Hence, $f_c$ is a strictly increasing homeomorphism of $\R$. The uniform lower bound on the slopes also shows that both $f_c$ and the inverse $f_c^{-1}$ depend continuously on both $c$ and its argument. It follows that $\Phi_V$ is a homeomorphism of $\R^3$. Continuity along the axis follows from
\[
T(\eps,\alpha)\longrightarrow0
\qquad\text{uniformly in $\alpha$ as }\eps\to0.
\]

Next, note that by \eqref{eq:kappadef}, \eqref{eq:kappa_ball}, \eqref{eq:T}, and the properties of $\mu$, every point at which $\Phi_V$ differs from the identity, as well as its image under $\Phi_V$, lies in $B^3(V,r)$. This ensures that by setting
\[
K(V):=\{(\eps,\alpha,h):\eps<\kappa\}\cap \mathrm{int}\, B^3(V,r),
\]
we have $\Phi_V(x)=x$ for every $x\notin K(V)$. Moreover, $T(0,\alpha)=0$ implies that $\Phi_V(V)=V$. Thus, condition~(a) is satisfied, and $K(V)$ serves as the bending neighborhood of the softening transformation. Furthermore, $\Phi_V$ is globally injective.

It remains to examine the effect of $\Phi_V$ on the vertex figure $\nu(V)$. By \eqref{eq:kappadef}, every point $(\eps,\alpha,h)\in\nu(V)$ with $\eps<\kappa$ lies in $B^3(V,r/4)$, and therefore $\mu(h)=1$. Accordingly, on the modified part of $\nu(V)$, the map $\Phi_V$ is given by the axis-parallel displacement
\[
(\eps,\alpha,h)
\longmapsto
\bigl(
\eps,\alpha,h+T(\eps,\alpha)
\bigr).
\]

Since $T$ is smooth away from the origin and the relative interiors of the edges and faces of $\nu(V)$ are smooth curves and surfaces, respectively, the restriction of $\Phi_V$ to each of these relative interiors is smooth. Its inverse on the corresponding image is obtained by subtracting $T(\eps,\alpha)$ and is therefore smooth as well. Hence, these restrictions are smooth diffeomorphisms, and condition~(b) of Definition~\ref{def:Vlagyitas} is satisfied.

The half-tangent criterion of  condition (c) follows from the fact that $\displaystyle \lim_{x \to 0^+} \varphi'(x)=\infty$. Consequently, the half-tangent of the curve $\Phi_V(\overline{Vv})$ at $V$ lies on the axis $t_V$.

Let $C$ be a cell of $\M$ containing $V$. By Lemma~\ref{lemma:policolor}, the face of the spherical vertex graph corresponding to $C$ contains vertices of both signs. Hence, there exist edges $e_i$ and $e_j$ of $C$ incident to $V$ such that $\sigma_i=+1$ and $\sigma_j=-1$. By the verification of condition~(c), the half-tangents of $\Phi_V(e_i)$ and $\Phi_V(e_j)$ at $V$ are the two opposite rays of the axis $t_V$, and hence their union is the complete line $t_V$.

Note that $t_V$ cannot be parallel to any chord connecting distinct points of edges in $\nu(V)$, since otherwise it would be coplanar with the segment connecting the end vertices of the relevant segments or, in case the two points lie on the same edge,  would contain a vertex. Therefore, Lemma~\ref{lemma:smoothjoin} can be applied to a suitable reparametrization of $e_i \cup e_j$, showing that $\Phi_V(e_i) \cup \Phi_V(e_j) $ is indeed a smooth curve, which establishes condition (d). Moreover, by pairing an edge with one of the opposite-colored edges and applying  Lemma~\ref{lemma:smoothjoin} again, we derive the smooth curve part of (c) as well. 
This finishes the proof that $\Phi_V$ is a local softening homeomorphism at $V$.

Finally, since the balls $B^3(V,r)$ associated with distinct nodes are pairwise disjoint, each $\Phi_V$ is the identity outside its bending neighborhood, and the family of bending neighborhoods is locally finite, the local softening transformations combine into a global homeomorphism on $\R^3$. Hence, by Definition~\ref{def:Mlagyitas}, the tiling $\M$ can be completely softened.
\end{proof}

The effect of the softening procedure on the intersection of a face with a cylindrical surface is illustrated in Figure~\ref{abra:simitas}.

\begin{figure}[H]
    \centering
    \includegraphics[width=1\linewidth]{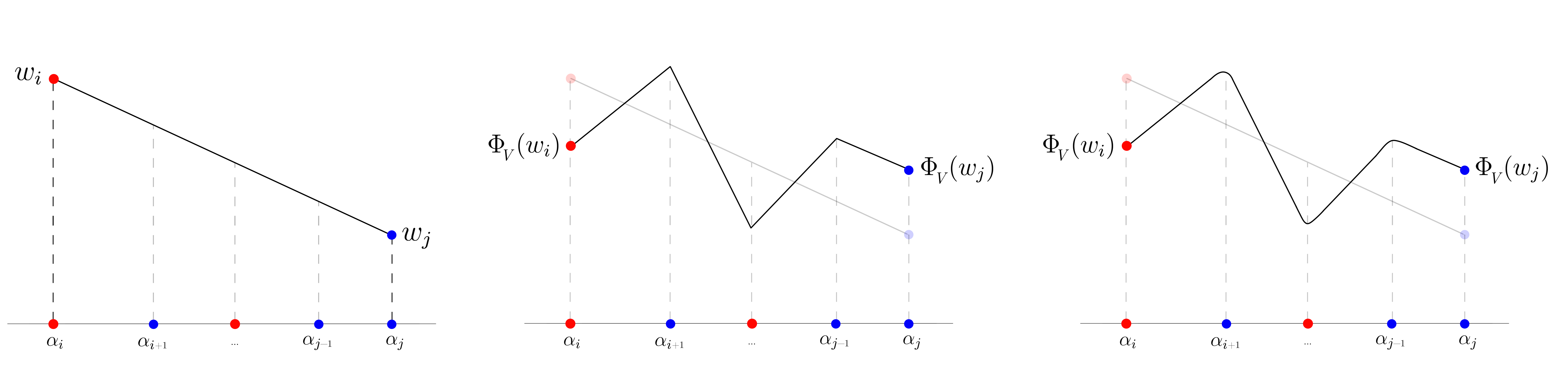}
    \caption{The smoothing procedure on the cylindrical surface $H_\eps$: cross-section of a face and its transformed images using translations by $\tau$ and $T$}
    \label{abra:simitas}
\end{figure}

Next, we extend the algorithm to locally polyhedral tilings, introduced in Definition~\ref{def:lokpol}.

\begin{proof}[Proof of Theorem~\ref{thm:fotetel} for locally polyhedral tilings.]
Let $\M$ be a locally polyhedral tiling. For every node $V$, choose a sufficiently small radius $r_V>0$ such that the balls $B^3(V,r_V)$ are pairwise disjoint, meet no edge or face not incident to the corresponding node, and are contained in neighborhoods witnessing the local polyhedrality of $\M$.

Fix a node $V$ and write $r=r_V$. Let $e_1,\ldots,e_n$ be the edges incident to $V$, and let $u_i$ be the direction of the half-tangent of $e_i$ at $V$. Since there are only finitely many edges incident to $V$, by applying an initial local deformation, we may assume that the $u_i$ are pairwise distinct. 

The spherical vertex graph at $V$ is combinatorially equivalent to the edge graph of a polyhedron. Hence, by Lemma~\ref{lemma:policolor}, we may assign signs $\sigma_i\in\{+1,-1\}$ to the edges in such a way that the boundary of every cell containing $V$ contains edges of both signs.

Choose an axis $t_V$ through $V$ which is not parallel to any $u_i$ and is not contained in any of the planes spanned by two distinct directions $u_i$ and $u_j$. In particular, the azimuth angles of the vectors $u_i$ with respect to $t_V$ are pairwise distinct. By decreasing $r$ if necessary, we may assume that the portions $e_i\cap B^3(V,r)$ lie in pairwise disjoint angular sectors, none of their tangents is parallel to $t_V$, and $t_V$ is not parallel to any chord joining two distinct points of these edge portions.

We use the cylindrical coordinates and notation of the polyhedral case, with axis $t_V$ and radius $r$. Since the half-tangent of $e_i$ is not parallel to $t_V$, the distance from $t_V$ is a regular local parameter on $e_i$ near $V$. Thus, after choosing $\kappa>0$ sufficiently small, every cylindrical surface $H_\eps$, $0<\eps\leq\kappa$, intersects each $e_i$ in exactly one point. Let $\alpha_i(\eps)$ and $h_i(\eps)$ denote its azimuth angle and axial coordinate. These functions are $C^2$ on $(0,\kappa]$, and the functions $\alpha_i$ extend continuously to $\eps=0$. Note that the linear ordering of these angles is independent of $\varepsilon$ for  $\varepsilon \in [0,\kappa]$. By decreasing $\kappa$ if necessary, the angular distances between consecutive functions $\alpha_i(\eps)$ are bounded below by a positive constant. We also require that $\kappa$ satisfies \eqref{eq:kappa_ball} and 
\[
\nu(V) \cap \bigl\{(\eps, \alpha, h): 0 \leq \eps \leq \kappa \bigr\} \subset B^3 \bigl( V, \frac r 4 \bigr).\]

We now repeat the construction from the polyhedral case, replacing the fixed angles $\alpha_i$ by the moving angles $\alpha_i(\eps)$. Instead of formula \eqref{eq:tea} we define 
\[
\tau\bigl(\eps,\alpha_i(\eps)\bigr)
=
\frac r4 \cdot \sigma_i \cdot
\varphi\left(\frac{\eps}{\kappa}\right),
\]
for $0<\eps\leq\kappa$, and extend $\tau(\eps,\alpha)$ linearly in the angular variable between consecutive values. Its smooth approximation $T(\eps,\alpha)$ is constructed as before, using the circular distance from the curves $\alpha = \alpha_i(\eps)$.

Since the functions $\alpha_i(\eps)$ are $C^2$ and their mutual angular distances are bounded below by a positive constant, the smoothing construction from the polyhedral case, carried out using the signed circular angular difference from the curves $\alpha=\alpha_i(\eps)$, allows the corresponding smoothing neighborhoods to be chosen pairwise disjoint. The resulting function $T$ is $C^2$ away from the axis, preserves the prescribed values on the edges, and satisfies
\[
\left|T(\eps,\alpha)\right|
\leq
\frac r4\,
\varphi\left(\frac{\eps}{\kappa}\right).
\]
In particular, $T(\eps,\alpha)\longrightarrow 0$ uniformly in $\alpha$ as $\eps\to0^+$.

Let $\mu$ be the cutoff function used in the polyhedral case and define the local homeomorphism $\Phi_V(\eps,\alpha,h)$ by \eqref{eq:phi0def}.
The proof is completed as in the polyhedral case. Note that after rescaling, Lemma~\ref{lemma:smoothjoin} applies in this more general setting as well, since the choice of $t_V$ guarantees its no-chord assumption. Thus, conditions (c) and (d) are established as before. 

Performing the construction at every node and combining the local maps, whose supports are pairwise disjoint and locally finite, yields a global homeomorphism of $\R^3$ which completely softens $\M$.
\end{proof}

\begin{remark}
As pointed out by one of the referees, the edge-bending algorithm described above is a special, explicitly described  case  of a more general scheme. For simplicity, assume that $V$ is the origin $o$. Given the two-coloring of the vertices of $G$, we consider a smooth family of diffeomorphisms  $f_\eps(u)$, $u\in S^2$ for $\eps \in (0,r]$ that is the identity on $S^2$ for $\eps = r$ and has a singular limiting map at $\eps=0$ that sends the direction vectors of blue vertices to the north pole and the directions of red vertices to the south pole. The local softening homeomorphism $\Phi$ is defined as $\Phi(o)=o$ and 
\[
\Phi(x)= |x | \cdot f_{s(|x|)}\left(\frac {x}{|x|}\right)
\]
for $|x| \in (0, r]$, where  $s \in C^2([0,r])$ is a strictly increasing function that satisfies $s(0)=s'(0)=0$, $s(r)=r$, and $s'(r)= s''(r)=0$. The reparametrization by $s$ ensures that the deformation joins smoothly to the rest of the vertex figure along the boundary of the bending neighborhood. However, to ensure that the union of two bent edges corresponding to vertices of different colors is indeed a smooth curve, the limiting angular map of the family $f$ must satisfy further flatness conditions. Our approach provides a specific, algorithmic realization of this general construction. 
\end{remark}

\section{Examples} 

We conclude the article by illustrating the softening procedure in some special cases.

\subsection{A softened cube tiling}
In the classical cube tiling, the vertex polyhedra are regular octahedra; see Figure~\ref{fig:csucspolieder}. The effect of the softening algorithm on a vertex figure is illustrated in Figure~\ref{fig:oktaederlagyitas}, while the softened cube cells appear in Figure~\ref{fig:lagycella}. In our setting, the softened cube demonstrates that all spikes can be removed while monohedrality is preserved.

\begin{figure}[H]
    \centering
    \includegraphics[width=0.8\linewidth]{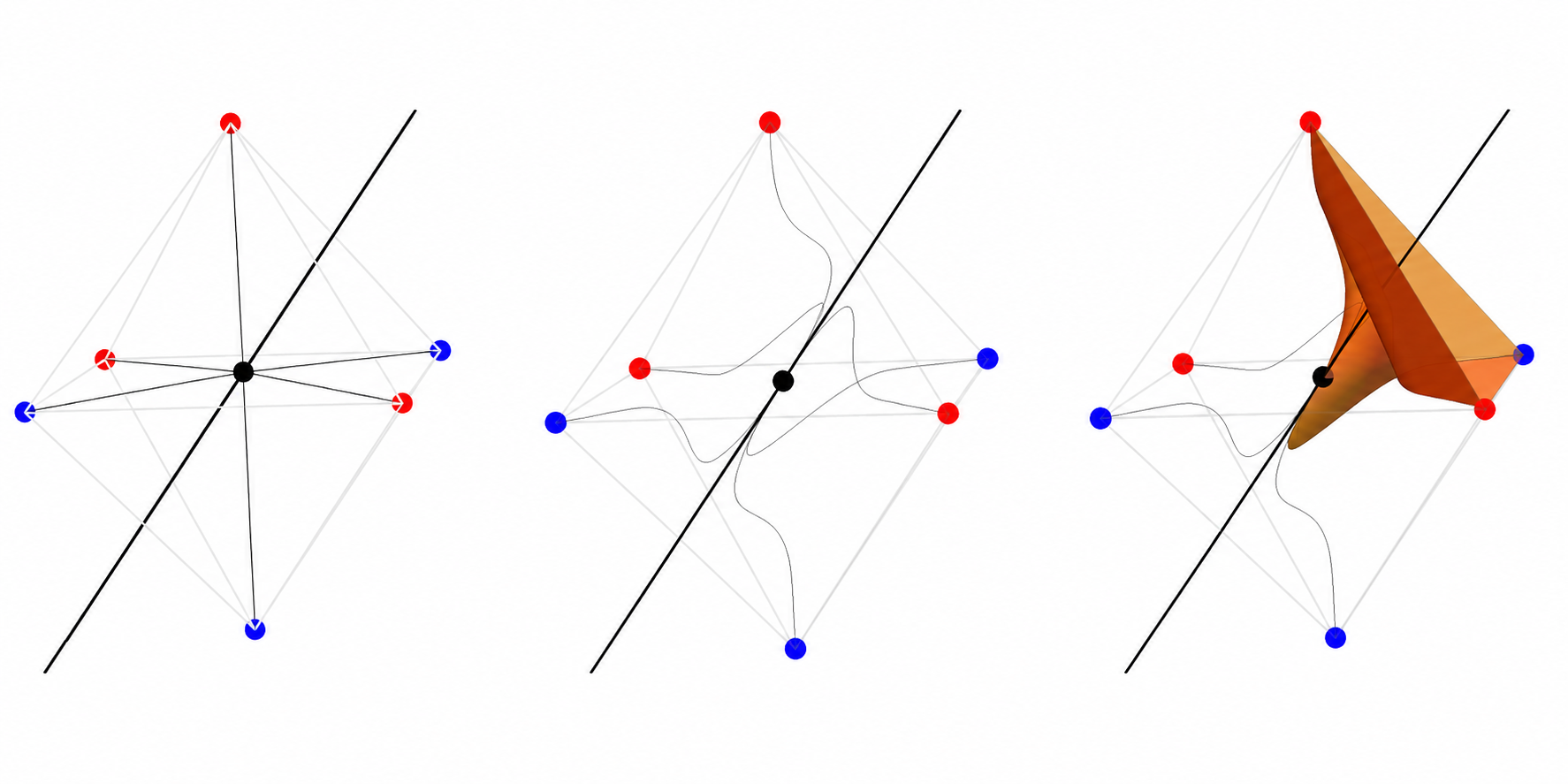}
    \caption{Vertex figure in a softened cube tiling}
    \label{fig:oktaederlagyitas}
\end{figure}

\begin{figure}[h]
    \centering
    \includegraphics[width=1\linewidth]{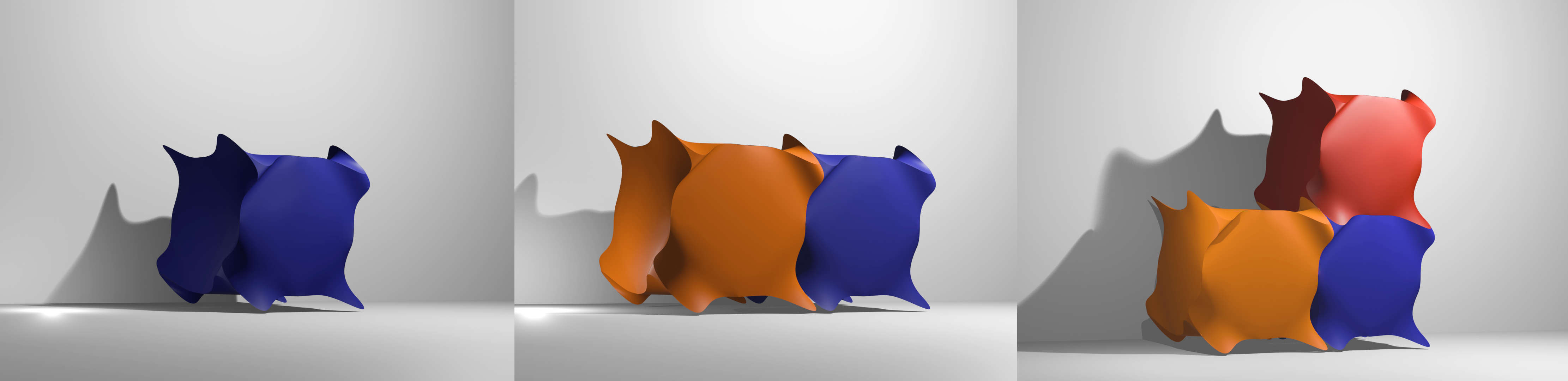}
    \caption{Soft cubes}
    \label{fig:lagycella}
\end{figure}

\subsection{Disconnected color classes}
A crucial requirement of the edge-bending algorithm developed by Domokos, Goriely, G.~Horváth, and Regős~\cite{DOM2024} is that the vertices of each vertex polyhedron can be assigned one of two colors such that no face is monochromatic and at least one color class is edge-connected; see Theorem~\ref{thm:DOM24}. In contrast, our softening method does not require edge-connectedness and therefore gives rise to new types of soft vertices.

We illustrate this phenomenon with a vertex figure that is combinatorially equivalent to a cube. In the 2-coloring in Figure~\ref{fig:kockaszin}, both color classes are disconnected. Theorem~\ref{thm:fotetel} is applicable even in this scenario, resulting in the softened node shown in Figure~\ref{fig:kockacsucsok}.  The meeting of two cells at this node is illustrated in Figure~\ref{fig:kockabelső}.
\begin{figure}[H]
    \centering
    \includegraphics[width=0.8\linewidth]{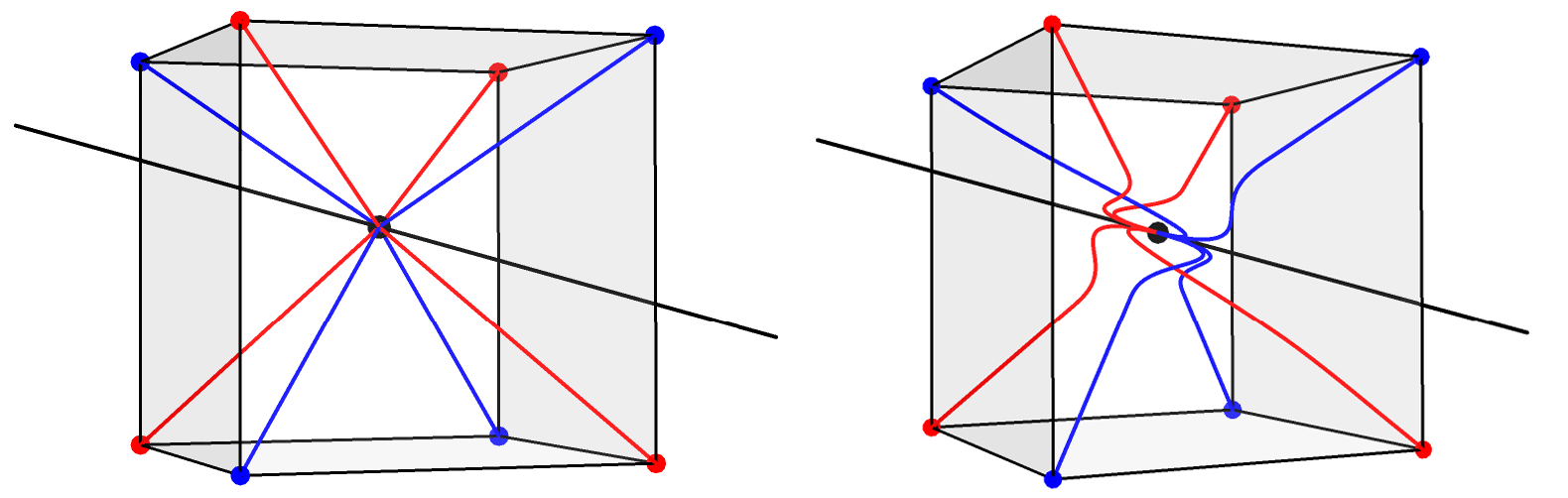}
    \caption{A vertex figure with disconnected color classes}
    \label{fig:kockaszin}
\end{figure}

\begin{figure}[H]
    \centering
    \includegraphics[width=0.8\linewidth]{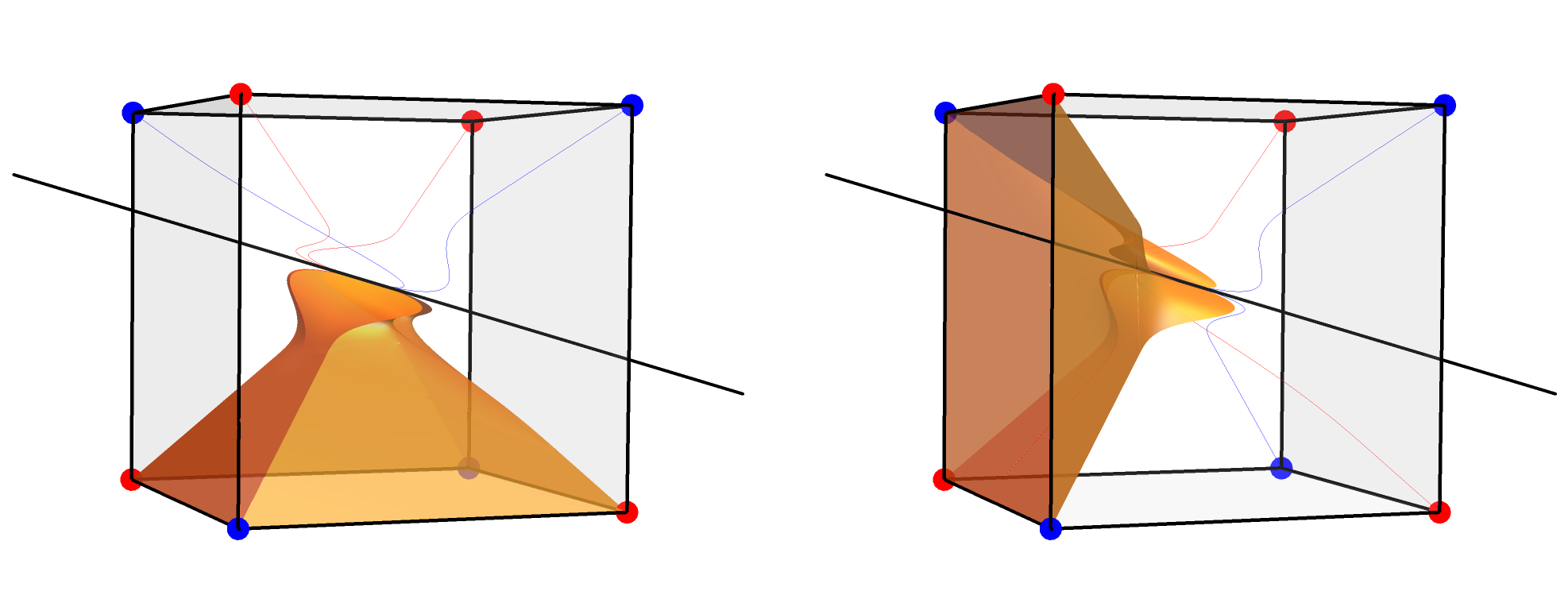}
    \caption{Cells arising from disconnected color classes} 
    \label{fig:kockacsucsok}
\end{figure}

\begin{figure}[H]
    \centering
    \includegraphics[width=0.4\linewidth]{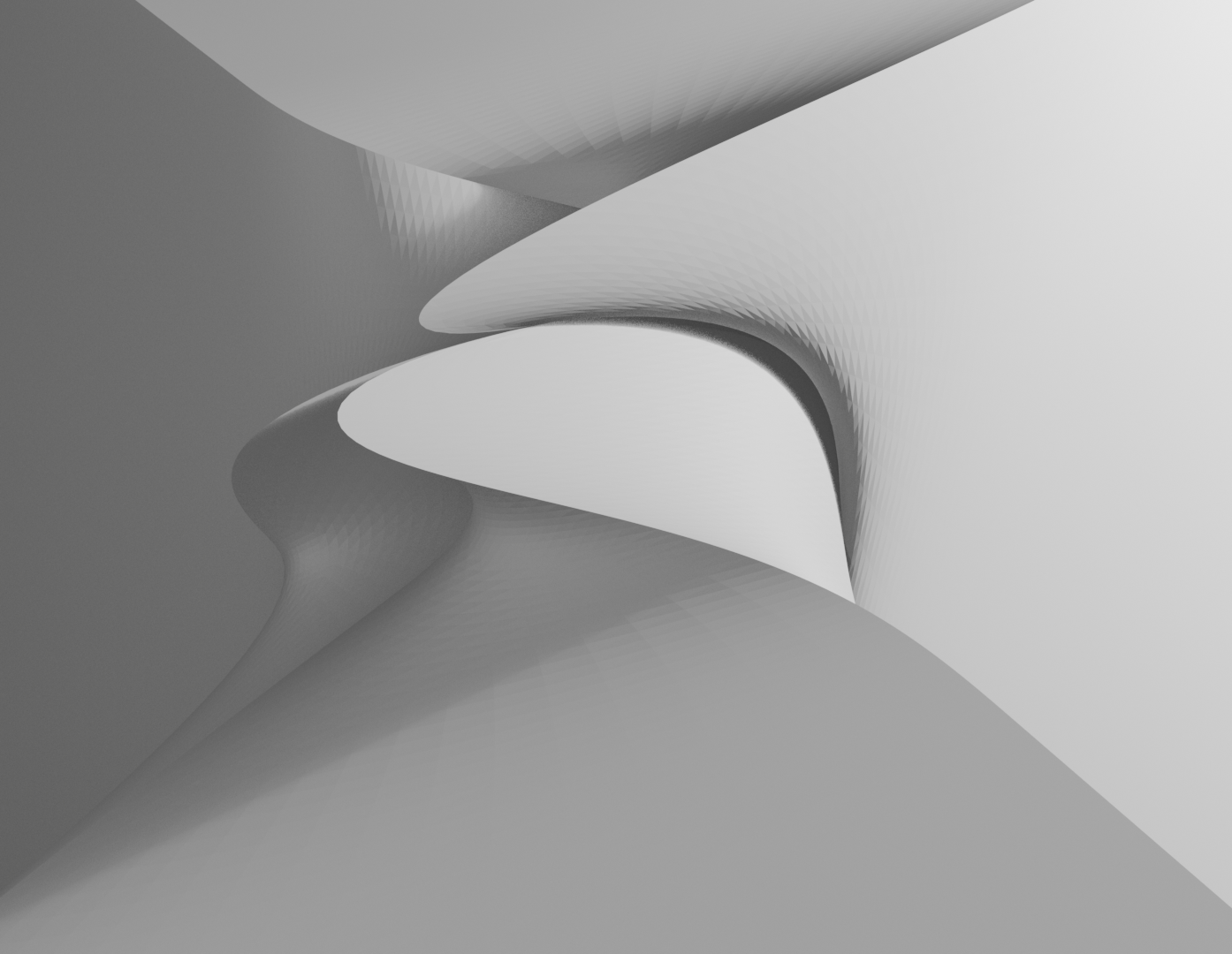}
    \caption{An embracing pair of soft cells}
    \label{fig:kockabelső}
\end{figure}

\subsection{A locally polyhedral tiling}

Figure~\ref{fig:orsolakhely} shows a locally polyhedral tiling obtained by inflating certain edges of a cube tiling into four-faced spindles (which are not polyhedra, but the resulting vertex figures are still locally polyhedral). The softened counterpart of such a spindle is shown in Figure~\ref{fig:ize}.

\begin{figure}[H]
    \centering
    \includegraphics[width=0.75\linewidth]{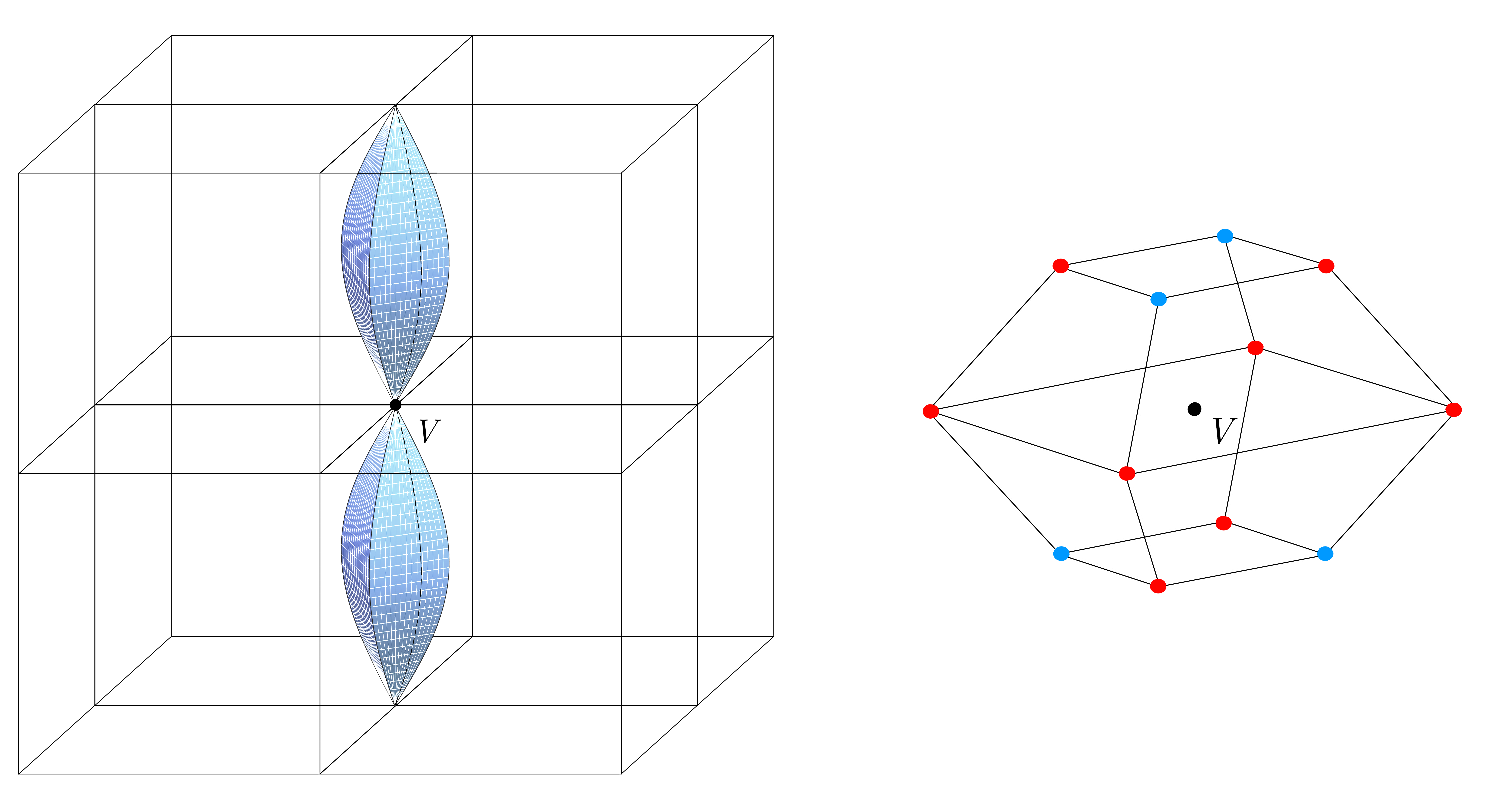}
    \caption{A locally polyhedral tiling and the corresponding well-colored vertex polyhedron}
    \label{fig:orsolakhely}
\end{figure}

\begin{figure}[H]
    \centering
    \includegraphics[width=0.33\linewidth]{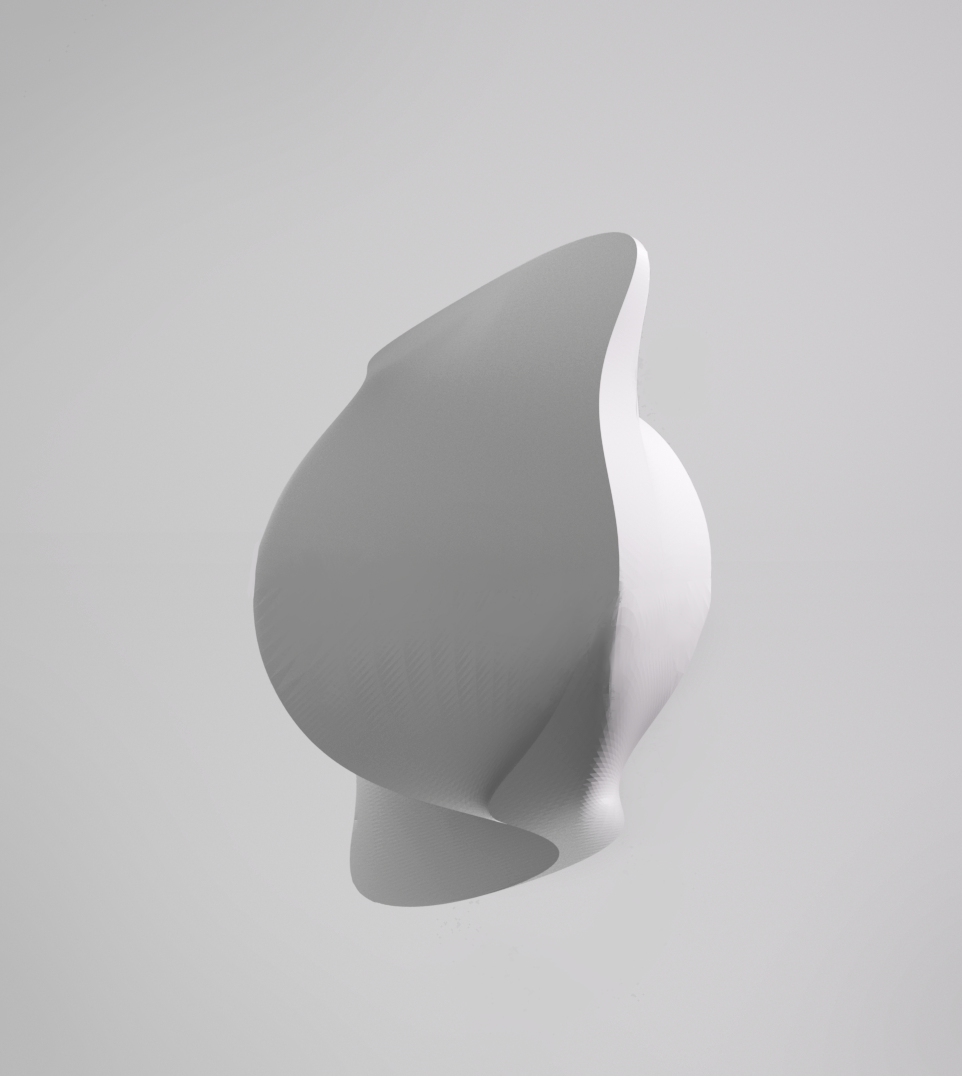}
    \caption{A softened four-faced spindle}
    \label{fig:ize}
\end{figure}


\subsection{Limitations of the edge-bending algorithm for polyhedric tilings}
A simple generalization of the polyhedric tiling in Figure~\ref{fig:paplan} shows the limitation of the softening method. In Figure~\ref{fig:ellenpelda}, we give an example of a vertex figure of the node $V$ in a polyhedric tiling that cannot be softened by the edge-bending algorithm, since in any 2-coloring of its vertices, one of the faces bounded by a pair of parallel edges between two of the vertices $A$, $B$ and $C$ is forced to be monochromatic.

\begin{figure}[H]
    \centering
    \includegraphics[width=0.33\linewidth]{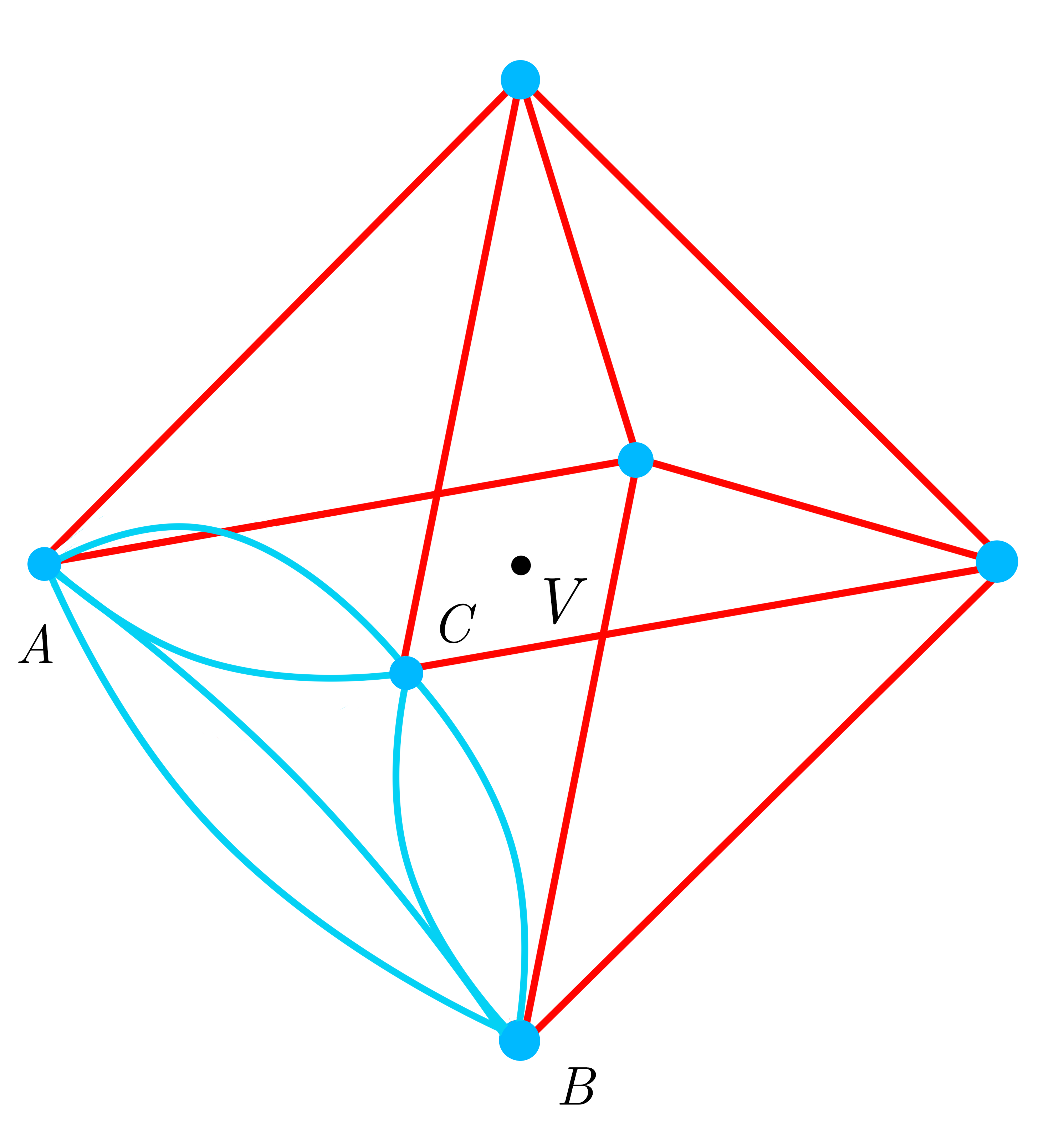}
    \caption{A polyhedric vertex figure for which edge bending fails}
    \label{fig:ellenpelda}
\end{figure}

\section{Acknowledgements}

The authors are grateful to Gábor Fejes Tóth for valuable advice and especially indebted to the anonymous referees for their insightful and constructive comments, which greatly improved both the presentation and clarity of the results.

\medskip

\bibliographystyle{abbrv}
\bibliography{references_softcells}

\bigskip

\noindent
{\sc Gergely Ambrus}
\smallskip

\noindent
{\small 
{\em Bolyai Institute, University of Szeged, Hungary \\ and\\ 
 Alfréd Rényi Institute of Mathematics, Budapest, Hungary
 }}
\smallskip

\noindent
e-mail address: \texttt{ambrus@server.math.u-szeged.hu, ambrus@renyi.hu}

\bigskip

\noindent
{\sc Dorottya Dancsó}
\smallskip

\noindent
{\em Bolyai Institute, University of Szeged, Hungary  
  }
\smallskip

\noindent
e-mail address: \texttt{Dancso.Dorottya.Maria@stud.u-szeged.hu}

\end{document}